\documentclass{birkjour}

\usepackage[utf8]{inputenc}
\usepackage[T1]{fontenc}

\usepackage{graphicx}		
\usepackage{mathscinet}		
\usepackage{enumitem}		
\usepackage{color}			
\usepackage{hyperref}		
\hypersetup{hidelinks}		
\usepackage{pdflscape}		

\usepackage{amssymb, amsmath,amsfonts,latexsym}
\usepackage{empheq}

\usepackage{booktabs}
\usepackage{tabulary}
\usepackage{tabularx}
\usepackage{tablefootnote}
\usepackage{longtable}

\usepackage[font=small,floatrowsep=qquad,captionskip=5pt]{floatrow}
\newcommand{\N}{{\mathbb{N}}}  
\newcommand{\Z}{{\mathbb{Z}}}  
\newcommand{\R}{{\mathbb{R}}}  
\newcommand{\Czi}{{C^{\infty}_{0}}}  
\newcommand{\Sw}{\mathcal{S}}	

\newcommand{\db}{\normalfont{\text{\dbar}}}
\newcommand{\rmd}{\mathrm{d}}
\newcommand{\I}{\mathrm{i}}

\newcommand{\jbl}{\langle}
\newcommand{\jbr}{\rangle}

\newcommand{\vp}{\varphi}
\newcommand{\ve}{\varepsilon}

\newcommand{\Osi}{\mathrm{Os-}\int\limits_{\R^n}}
\newcommand{\Osii}{\mathrm{Os-}\iint\limits_{\R^{2n}}}
\newcommand{\Sy}{S}
\newcommand{\OPS}{\Psi}
\newcommand{\OP}{Op}
\newcommand{\jxi}{\jbl\xi\jbr}
\newcommand{\B}{B}

\newcommand{\W}{\mathcal W}
\renewcommand{\Re}{\operatorname{Re}}

\DeclareMathOperator*{\supp}{supp}

\newtheorem{Theorem}{Theorem}[section]

\newtheorem{Proposition}[Theorem]{Proposition}
\theoremstyle{definition}
\newtheorem{Definition}[Theorem]{Definition}
\theoremstyle{remark}
\newtheorem{Remark}[Theorem]{Remark}
\newtheorem*{Example}{Example}
\numberwithin{equation}{section}

\newcommand*\xbar[1]{%
	\hbox{%
		\vbox{%
			\hrule height 0.5pt 
			\kern0.5ex
			\hbox{%
				\kern-0.1em
				\ensuremath{#1}%
				\kern-0.1em
			}%
		}%
	}%
}

\allowdisplaybreaks

\begin{document}
	

	\title[Strictly Hyperbolic Equations with Coefficients Low-Regular in Time]{Strictly Hyperbolic Equations with Coefficients Low-Regular in Time and Smooth in Space}
	\author[Cicognani]{Massimo Cicognani}
	\address{%
	Universita di Bologna, Dipartimento di Matematica\\Piazza di Porta San Donato, 5,\\40126 Bologna, Italy	
	}
	\email{massimo.cicognani@unibo.it}
	
	\author[Lorenz]{Daniel Lorenz}
	\address{%
		TU Bergakademie Freiberg, Faculty of Mathematics and Computer Science\\Institute of Applied Analysis\\Pr{\"u}ferstraße 9,\\09599 Freiberg, Germany
	}
	\email{daniel.lorenz@math.tu-freiberg.de}
	
	
	\begin{abstract}
		We consider the Cauchy problem for strictly hyperbolic $m$-th order partial differential equations with coefficients low-regular in time and smooth in space. It is well-known that the problem is $L^2$ well-posed in the case of Lipschitz continuous coefficients in time, $H^s$ well-posed in the case of Log-Lipschitz continuous coefficients in time (with an, in general, finite loss of derivatives) and Gevrey well-posed in the case of H{\"o}lder continuous coefficients in time (with an, in general, infinite loss of derivatives). Here, we use moduli of continuity to describe the regularity of the coefficients with respect to time, weight sequences for the characterization of their regularity with respect to space and weight functions to define the solution spaces. We establish sufficient conditions for the well-posedness of the Cauchy problem, that link the modulus of continuity and the weight sequence of the coefficients to the weight function of the solution space. The well-known results for Lipschitz, Log-Lipschitz and H{\"o}lder coefficients are recovered.
	\end{abstract}

	\maketitle

	\section{Introduction}\label{INTRO}
	
	We consider the strictly hyperbolic Cauchy problem
	\begin{equation}\label{SH:BM:CauchyProblem}
	\begin{aligned}
	D_t^m u &= \sum\limits_{j=0}^{m-1} A_{m-j}(t,\,x,\,D_x) D_t^j u  + f(t,x),\\
	D_t^{k-1} u(0,x) &= g_k(x), \quad (t,x) \in [0, T]\times \R^n,\,k = 1,\,\ldots,\,m,
	\end{aligned}
	\end{equation}
	where
	\begin{equation*}
		A_{m-j}(t,\,x,\,D_x) = \sum\limits_{|\gamma| + j = m} a_{m-j, \gamma}(t,\,x) D_x^\gamma + \sum\limits_{|\gamma| + j \leq m-1} a_{m-j,\gamma}(t,\,x) D_x^\gamma,
	\end{equation*}
	and $D_t = \frac{1}{\I} \partial_t$, $D_x^\gamma = (\frac{1}{\I} \partial_x)^\gamma$, as usual.
	We are interested in well-posedness results for the above Cauchy problem, when the regularity of the coefficients with respect to time is Lipschitz or below Lipschitz. We obtain a sufficient result for well-posedness which links the regularity of the coefficients in time to their regularity in space and the possible solution spaces.
	
	It is well-known that the strictly hyperbolic Cauchy problem is not well-posed in $C^\infty, H^\infty$, respectively, if the regularity of the coefficients in time is lower than Lipschitz. Usually, one has to compensate for the low regularity in time by assuming higher regularity in space. Colombini et al.~\cite{Colombini.1979} proved this fact for second-order equations with H{\"o}lder continuous coefficients which only depend on $t$.
	Considering operators whose coefficients are H{\"o}lder continuous in time and Gevrey in the spatial variables, Nishitani~\cite{Nishitani.1983} and Jannelli~\cite{Jannelli.1985} were able to extend the results of \cite{Colombini.1979}.
	
	Colombini and Lerner~\cite{Colombini.1995} considered second-order operators with coefficients which are Log-Lipschitz in time and $C^\infty$ in space (with some additional $L^\infty$ conditions). They proved well-posedness in Sobolev-spaces (with finite loss of derivatives) and established the Log-Lipschitz regularity as the natural threshold beyond which no Sobolev well-posedness could be expected.
	
	Cicognani~\cite{Cicognani.1999} extended the results of \cite{Colombini.1995, Jannelli.1985, Nishitani.1983} to equations of order $m$ and considered Log-Lipschitz and H{\"o}lder continuous coefficients, which also depend on $x$.
	
	For second order equations Cicognani and Colombini~\cite{Cicognani.2006} provided a classification, linking the loss of derivatives to the modulus of continuity of the coefficients with respect to time.
	
	In this paper, we consider the strictly hyperbolic Cauchy problem \eqref{SH:BM:CauchyProblem}, where we assume that the coefficients $a_{m-j,\,\gamma}$ satisfy
	\begin{equation*}
		\big|D_x^\beta a_{m-j,\,\gamma}(t,\,x) - D_x^\beta a_{m-j,\,\gamma}(s,\,x)\big| \leq C K_{|\beta|} \mu(|t-s|),\, 0 \leq |t-s| \leq 1, \,x\in\R^n,
	\end{equation*}
	where $\mu$ is a modulus of continuity describing their regularity in time and $K_p$ is a weight sequence describing their regularity in space. To describe the regularity of the initial data, the right hand side and the solution, we use the spaces
	\begin{equation*}
		H^\nu_{\eta,\,\delta}(\R^n) = \big\{f \in \Sw^\prime(\R^n);\, e^{\delta\eta(\jbl D_x \jbr)} f(x) \in H^\nu(\R^n)\big\},
	\end{equation*}
	where $H^\nu(\R^n) = H^{\nu,\,2}(\R^n)$ denotes the usual Sobolev-spaces and $\delta > 0$ is a constant.
	Under suitable conditions we are able to prove that our problem is well-posed and that we have a global (in time) solution which belongs to the space
	\begin{equation*}
		\bigcap\limits_{j = 0}^{m-1} C^{m-1-j}\big([0,T];\, H^{\nu+j}_{\eta,\,\delta}(\R^n)\big).
	\end{equation*}
	As to be expected from the know results of the above-mentioned authors, the modulus of continuity $\mu$ is linked to the weight function $\eta$. In this paper, we describe how $\mu$ and $\eta$ are related to 
	each other and give sufficient conditions for the well-posedness of problem \eqref{SH:BM:CauchyProblem} which link $\mu$ to $\eta$ and the sequence $K_p$.
	
	The paper is organized as follows: Section~\ref{DEF} reviews some definitions and useful propositions related to moduli of continuity. It also provides an introduction to the pseudodifferential calculus used in this paper. Section~\ref{RESULTS} states the main results of this paper and also features some examples and remarks. Finally, in Section~\ref{PROOF} we proceed to prove the theorems of the previous section.

	\section{Definitions and Useful Propositions}\label{DEF}
	
	Let $x = (x_1,\,\,\ldots,\,x_n)$ be the variables in the $n$-dimensional Euclidean space $\R^n$ and by $\xi = (\xi_1,\,\ldots,\,\xi_n)$ we denote the dual variables. Furthermore, we set $\jbl\xi\jbr^2 = 1 + |\xi|^2$.
	We use the standard multi-index notation. Precisely, let $\Z$ be the set of all integers and $\Z_+$ the set of all non-negative integers. Then $\Z^n_+$ is the set of all $n$-tuples $\alpha = (\alpha_1,\,\ldots,\,\alpha_n)$ with $a_k \in \Z_+$ for each $k = 1,\,\ldots,\,n$. The length of $\alpha \in \Z^n_+$ is given by $|\alpha| = \alpha_1 + \ldots + \alpha_n$.\\
	Let $u = u(t,\,x)$ be a differentiable function, we then write
	\begin{equation*}
		u_t(t,\,\xi) = \partial_t u (t,\,x) = \frac{\partial}{\partial t} u(t,\,x),
	\end{equation*}
	and
	\begin{equation*}
		\partial_x^\alpha u (t,\,x) = \left(\frac{\partial}{\partial x_1}\right)^{\alpha_1} \ldots\left(\frac{\partial}{\partial x_n}\right)^{\alpha_n} u(t,\,x).
	\end{equation*}
	Using the notation $D_{x_j} = -\I \frac{\partial}{\partial x_j}$, where $\I$ is the imaginary unit, we write also
	\begin{equation*}
		D_x^\alpha = D_{x_1}^{\alpha_1} \cdots D_{x_n}^{\alpha_n}.
	\end{equation*}
	Similarly, for $x\in \R^n$ we set
	\begin{equation*}
		x^\alpha = x_1^{\alpha_1} \cdots x_n^{\alpha_n}.
	\end{equation*}
	In the context of pseudodifferential operators and the related symbol calculus, we sometimes use the notation
	\begin{equation*}
		a^{(\alpha)}_{(\beta)}(x,\,\xi) = \partial_\xi^\alpha D_x^\beta a(x,\,\xi).
	\end{equation*}

	Let $f$ be a continuous function in an open set $\Omega \subset \R^n$. By $\supp f$ we denote the support of $f$, i.e. the closure in $\Omega$ of $\{x \in \Omega,\,f(x) \neq 0\}$. By $C^k(\Omega)$, $0 \leq k \leq \infty$, we denote the set of all functions $f$ defined on $\Omega$, whose derivatives $\partial^\alpha_x f$ exist and are continuous for $|\alpha| \leq k$. By $\Czi(\Omega)$ we denote the set of all functions $f \in C^\infty(\Omega)$ that have compact support in $\Omega$. The Sobolev-space $H^{k,p}(\Omega)$ consists of all functions that are $k$ times differentiable in Sobolev-sense and have (all) derivatives in $L^p(\Omega)$.
	
	For two functions $f=f(x)$ and $g=g(x)$ we write
	\begin{align*}
		f(x) = o(g(x))\qquad \text{ if }\qquad \lim\limits_{x\rightarrow \infty} \frac{f(x)}{g(x)} = 0,
	\end{align*}
	and we use the notation
	\begin{equation*}
		f(x) = O(g(x))\qquad \text{ if }\qquad \limsup\limits_{x\rightarrow \infty} \frac{f(x)}{g(x)} \leq C.
	\end{equation*}
	We use $C$ as a generic positive constant which may be different even in the same line.

	Furthermore, we introduce the following spaces.
	\begin{Definition}
		Let $\eta$ be a real, smooth, increasing function, $\nu \in \R$ and $\delta > 0$. We define the space $H^\nu_{\eta,\,\delta} = H^\nu_{\eta,\,\delta}(\R^n)$ by
		\begin{equation*}
			H^\nu_{\eta,\,\delta}(\R^n) = \big\{f \in \Sw^\prime(\R^n);\, e^{\delta\eta(\jbl D_x \jbr)} f(x) \in H^\nu(\R^n)\big\},
		\end{equation*}
		where $H^\nu(\R^n)=H^{\nu,2}(\R^n)$ denotes the usual Sobolev-spaces.
	\end{Definition}
	\begin{Definition}
		Let $K_p$ be a positive, increasing sequence of real numbers. We define the space $\B_K^\infty = \B_K^\infty(\R^n)$ by
		\begin{equation*}
			\B_K^\infty(\R^n) = \big\{ f \in C^\infty(\R^n);\, \sup_{x\in\R^n} |D_x^\beta f(x)| \leq C K_{|\beta|}\, \text{for all } \beta \in \N^n \big\}.
		\end{equation*}
	
		By $\B^\infty=\B^\infty(\R^n)$ we denote the space of all smooth functions that have bounded derivatives.
	\end{Definition}

	\subsection{Moduli of Continuity}

	As explained in Section~\ref{INTRO}, we use moduli of continuity to describe the regularity of the coefficients with respect to time.
	Let us briefly recall what we understand by the term modulus of continuity.
	\begin{Definition}[Modulus of Continuity and $\mu$-Continuity]\label{SH:BM:RegTime:Definition:MOC}
		We call $\mu: [0,\,1] \rightarrow [0,\,1]$ a modulus of continuity, if $\mu$ is continuous, concave and increasing and satisfies $\mu(0) = 0$.
		A function $f \in C(\R^n)$ belongs to $C^\mu(\R^n)$ if and only if
		\begin{equation*}
			|f(x) - f(y)| \leq C \mu(|x-y|),
		\end{equation*}
		for all $x,\,y \in \R^n,\, |x-y| \leq 1$ and some constant $C$.
	\end{Definition}
	Typical examples of moduli of continuity are presented in the following table.
	\begin{center}
		\begin{tabulary}{\textwidth}{LL}
			\toprule
			modulus of continuity &commonly called\\
			\midrule
			$\mu(s) = s$	& Lipschitz-continuity\\[5pt]
			$\mu(s) = s \left(\log\left(\frac{1}{s}\right) + 1 \right)$	& Log-Lip-continuity\\[4pt]
			$\mu(s) = s \left(\log\left(\frac{1}{s}\right) + 1\right)\log^{[m]}\left(\frac{1}{s}\right)$ 	& Log-Log$^{[m]}$-Lip-continuity\\[5pt]
			$\mu(s) = s^\alpha,\quad \alpha\in(0,\,1)$	& H{\"o}lder-continuity \\[5pt]
			$\mu(s) = \left(\log\left(\frac{1}{s}\right) + 1 \right)^{-\alpha},\quad \alpha\in(0,\,\infty)$	& Log$^{-\alpha}$-continuity\\[5pt]
			\bottomrule
		\end{tabulary}
	\end{center}
	For convenience, we introduce the notion of weak and strong moduli of continuity (compared to the threshold of $\mu(s) = s \log(s^{-1})$).
	\begin{Definition}\label{SH:BM:ASSUME:ClassMOC}
		We call a given modulus of continuity $\mu$ strong, if
		\begin{equation*}
			\lim\limits_{s\rightarrow 0+} \frac{\mu(s)}{s \log(s^{-1})} \leq C,
		\end{equation*}i.e. functions belonging to $C^\mu$ are Log-Lip-continuous or more regular. Consequently, $\mu$ is called a weak modulus of continuity, if
		\begin{equation*}
			\lim\limits_{s\rightarrow 0+} \frac{s \log(s^{-1})}{\mu(s)} = 0,
		\end{equation*}i.e. functions belonging to $C^\mu$ are less regular than Log-Lip.
	\end{Definition}
	
	\subsection{Symbol Classes and Symbolic Calculus}
	
	We introduce the standard symbol classes of pseudodifferential operators following H{\"o}rmander~\cite{Hormander.2007}.
	
		\begin{Definition}[$\Sy^m_{\rho,\,\delta}$ and $\OPS^{m}_{\rho,\,\delta}$]\label{APP:PSEUDO:DEF:SmRD}\label{APP:PSEUDO:DEF:Sm}
			Let $m,\,\rho,\,\delta$ be real numbers with $0 \leq \delta < \rho \leq 1$. Then we denote by $\Sy^m_{\rho,\,\delta} = \Sy^m_{\rho,\,\delta}(\R^n\times\R^n)$ the set of all $a\in C^\infty(\R^n \times \R^n)$ such that for  all multi-indexes $\alpha,\,\beta$ the estimate
			\begin{equation*}
			|D_x^\beta \partial_\xi^\alpha a(x,\,\xi)| \leq C_{\alpha,\,\beta} (1 + |\xi|)^{m - \rho|\alpha| + \delta |\beta|},
			\end{equation*}
			is valid for all $x,\,\xi \in \R^n$ and some constant $C_{\alpha,\,\beta}$.
			We write $\Sy^{-\infty}_{\rho,\,\delta} = \bigcap_m \Sy^m_{\rho,\,\delta}$, $\Sy^\infty_{\rho,\,\delta} = \bigcup_m \Sy^m_{\rho,\,\delta}$.
			For a given $a=a(x,\,\xi) \in \Sy^m_{\rho,\,\delta}$, we denote by $\OP(a) = a(x,\,D_x)$ the associated pseudodifferential operator, which is defined as
			\begin{align*}
			a(x,\,D_x) u(x) = \int\limits_{\R^n} e^{\I x\cdot \xi} a(x,\,\xi)  \hat{u}(\xi) \db \xi
			= \Osii e^{\I (x-y)\cdot \xi} a(x,\,\xi)  u(y) \rmd y \db\xi,
			\end{align*}
			where $\db \xi = (2\pi)^{-n} \rmd \xi$ and $\Osii$ means the oscillatory integral.\\
			By $\OPS^m_{\rho,\,\delta} = \OPS^m_{\rho,\,\delta}(\R^n)$ we denote the set of all pseudodifferential operators that are associated to some symbol in $\Sy^m_{\rho,\,\delta}$.
			Conversely, for $a \in \OPS^{m}_{\rho,\,\delta}$, we denote by $\sigma(a) \in \Sy^{m}_{\rho,\,\delta}$ the associated symbol.
		\end{Definition}
		Analogously, we define sets of weighted symbols and associated operators.
		\begin{Definition}[$\Sy^{m,\,\omega}_{\rho,\,\delta}$ and $\OPS^{m,\,\omega}_{\rho,\,\delta}$]\label{APP:PSEUDO:DEF:SmORD}\label{APP:PSEUDO:DEF:SmO}
			Let $m,\,\rho,\,\delta$ be real numbers with $0 \leq \delta < \rho \leq 1$. Let $\omega$ be a non-decreasing, continuous function satisfying $\omega(s) = o(s)$ for $s \rightarrow +\infty$. Then we denote by $\Sy^{m,\,\omega}_{\rho,\,\delta} = \Sy^{m,\,\omega}_{\rho,\,\delta}(\R^n\ \times\R^n)$ the set of all $a\in C^\infty(\R^n \times \R^n)$ such that for all multi-indexes $\alpha,\,\beta$ the estimate
			\begin{equation*}
			|D_x^\beta \partial_\xi^\alpha a(x,\,\xi)| \leq C_{\alpha,\,\beta} (1 + |\xi|)^{m - \rho|\alpha| + \delta |\beta|} \omega(\jbl \xi \jbr),
			\end{equation*}
			is valid for all $x,\,\xi \in \R^n$ and some constant $C_{\alpha,\,\beta}$.
			The set $\OPS^{m,\,\omega}_{\rho,\,\delta} = \OPS^{m,\,\omega}_{\rho,\,\delta}(\R^n)$ is defined analogously to $\OPS^{m}_{\rho,\,\delta}$.
		\end{Definition}
	
	 As usual, if $(\rho,\,\delta) = (1,\,0)$ we omit them and just write $\Sy^m$,  $\Sy^{m,\,\omega}$, $\OPS^m$, $\OPS^{m,\,\omega}$.
	
	 Since $\OPS^{m,\,\omega} \subset \OPS^{m+1}$, the following composition result is obtained by straightforward computation.
	 \begin{Proposition}[Composition of $\OPS^m$ and $\OPS^{m,\,\omega}$]\label{APP:PSEUDO:CALC:COMP:SmO}
	 	Let $a_1 \in \Sy^{m_1}$ and $a_2 \in \Sy^{m_2,\,\omega}$, then as operators in $\Sw$ or $\Sw^\prime$
	 	\begin{align*}
	 		a_1(x,\,D_x) \circ a_2(x,\,D_x) &= b_1(x,\,D_x),\\
	 		a_2(x,\,D_x) \circ a_1(x,\,D_x) &= b_2(x,\,D_x),
	 	\end{align*}
	 	where $b_1,\,b_2 \in \Sy^{m_1 + m_2,\,\omega}$ have the asymptotic expansions
	 	\begin{align*}
	 		b_1(x,\,\xi) &\sim \sum\limits_\alpha \frac{1}{\alpha!}\partial_\xi^\alpha a_1(x,\,\xi) D_x^\alpha a_2(x,\,\xi),\\
	 		b_2(x,\,\xi) &\sim \sum\limits_\alpha \frac{1}{\alpha!}\partial_\xi^\alpha a_2(x,\,\xi) D_x^\alpha a_1(x,\,\xi).
	 	\end{align*}
	 \end{Proposition}
	
	Consider a pseudodifferential operator $a \in \OPS^m$ and a non-negative, increasing function $\psi\in C^\infty(\R^n)$. Throughout this paper we refer to the transformation
	\begin{equation*}
		a_{\psi}(x,\,D_x) = e^{\lambda \psi(\jbl D_x \jbr)} \circ a(x,\,D_x) \circ e^{-\lambda \psi(\jbl D_x \jbr)}
	\end{equation*}
	as conjugation, where $\lambda$ is a positive constant.

	\begin{Proposition}[Conjugation in $\OPS^m$]\label{APP:PSEUDO:CALC:CON}
		Let $a \in \OPS^m$ and let $\psi \in C^\infty(\R^n)$ be a non-negative, increasing function satisfying
		\begin{equation}\label{APP:PSEUDO:CALC:CON:EstEta}
		\bigg|\frac{\rmd^k}{\rmd s^k}\psi (s) \bigg| \leq C_k s^{-k} \psi(s),\qquad k \in \N, \, s \in \R.
		\end{equation}
		We fix a constant $\lambda > 0$. Then the symbol $a_{\psi}(x,\,\xi) = \sigma(a_{\psi}(x,\,D_x))$ of
		\begin{equation*}
			a_{\psi}(x,\,D_x) = e^{\lambda \psi(\jbl D_x \jbr)} \circ a(x,\,D_x) \circ e^{-\lambda \psi(\jbl D_x \jbr)}
		\end{equation*}
		satisfies
		\begin{equation}\label{APP:PSEUDO:CALC:CON:eq1}
		a_{\psi}(x,\,\xi) = a(x,\,\xi) + \sum\limits_{0 < |\gamma| < N} a_{(\gamma)}(x,\,\xi) \chi_\gamma(\xi) + r_N(x,\,\xi),
		\end{equation}
		where
		\begin{equation}\label{APP:PSEUDO:CALC:CON:eq2}
		\chi_\gamma(\zeta) = \frac{1}{\gamma!} e^{-\lambda \psi(\jxi)} \partial_\nu^\gamma\big(e^{\lambda \psi(\jbl\nu\jbr)}\big)\Big|_{\nu = \zeta},
		\end{equation}
		and
		\begin{equation}\label{APP:PSEUDO:CALC:CON:eq3}
			\begin{aligned}
				r_N(x,\,\xi) = \frac{N}{(2\pi)^n} \sum\limits_{|\gamma| = N} \bigg[&\Osii \int\limits_0^1 (1-\vartheta)^{N-1} e^{-\I y \zeta}\\
				&\times a_{(\gamma)}(x + \vartheta y,\,\xi) \chi_\gamma(\xi + \zeta) \rmd \vartheta \rmd y \rmd\zeta\bigg].
			\end{aligned}
		\end{equation}
		Furthermore, we have the estimate
		\begin{equation}
		|\partial_\xi^\alpha \chi_\gamma(\xi)| \leq C_{\alpha,\,\gamma} \jxi^{- |\alpha| - |\gamma|} (\psi(\jxi))^{|\gamma|} \label{APP:PSEUDO:CALC:CON:est1}
		\end{equation}
		for $\xi \in \R^n$ and $\alpha \in \N^n$.
	\end{Proposition}
	\begin{proof}
		Relations \eqref{APP:PSEUDO:CALC:CON:eq1}, \eqref{APP:PSEUDO:CALC:CON:eq2} and \eqref{APP:PSEUDO:CALC:CON:eq3} are derived in \cite{Kajitani.1983, Kajitani.2006} for the case $\psi(s) = s^\kappa$. Deriving these equations for general $\psi$ works just as described there. Estimate \eqref{APP:PSEUDO:CALC:CON:est1} is obtained by straightforward calculation. We have
		\begin{align*}
			|\partial_\xi^\alpha \chi_\gamma(\xi)|
			&=
			\frac{1}{\gamma!}\Big|\partial_\xi^\alpha \Big( e^{-\lambda \psi(\jxi)} \partial_\nu^\gamma\big(e^{\lambda \psi(\jbl \nu \jbr)}\big)\Big|_{\nu = \xi}\Big)\Big|\\
			&=
			\frac{1}{\gamma!}\Big|\partial_\xi^\alpha \Big( e^{-\lambda \psi(\jxi)} e^{\lambda \psi(\jxi)} Q_\gamma(\xi)\Big)\Big|\\	
			&\leq
			C_{\alpha,\,\gamma} \jxi^{- |\alpha| - |\gamma|} (\psi(\jxi))^{|\gamma|},
		\end{align*}
		where we used that
		\begin{equation*}
			\partial_\nu^\gamma\big(e^{\lambda \psi(\jbl\nu\jbr)}\big)\Big|_{\nu = \xi} = e^{\lambda \psi(\jxi)} Q_\gamma(\xi),
		\end{equation*}
		and applied \eqref{APP:PSEUDO:CALC:CON:EstEta} to
		\begin{equation*}
			Q_\gamma(\xi) = C_\gamma \psi^\prime(\jxi)^{|\gamma|} +R_\gamma(\xi),
		\end{equation*}
		where $R_\gamma(\xi)$ are lower order terms satisfying
		\begin{equation*}
			|R_\gamma(\xi)| \leq C_\gamma \psi^\prime(\jxi)^{|\gamma| - 1} \psi^{\prime\prime}(\jxi).
		\end{equation*}
		This yields estimate \eqref{APP:PSEUDO:CALC:CON:est1}.	
	\end{proof}
	\begin{Remark}
		By estimate \eqref{APP:PSEUDO:CALC:CON:est1} we are immediately able to conclude that $\chi_\gamma \in S^0$ for all $|\gamma| > 0$, since $\psi(\jxi) = o(\jxi)$.
	\end{Remark}
	
	Proposition~\ref{APP:PSEUDO:CALC:CON} does not provide an estimate for $r_N(x,\,\xi)$. In order to derive such an estimate we pose additional assumptions on the operator $a$ and the function $\psi$.
	
	\begin{Proposition}[Estimating the remainder $r_N(x,\,\xi)$]\label{APP:PSEUDO:CALC:CON:EstR}
		Take $a \in \OPS^m$ and $\psi \in C^\infty$ as in 	Proposition~\ref{APP:PSEUDO:CALC:CON}. Assume additionally that the symbol $a=a(x,\,\xi) \in \Sy^m$ is such that
		\begin{equation}\label{APP:PSEUDO:CALC:CON:EstR:EstA}
		|\partial_\xi^\alpha D_x^\beta a(x,\,\xi)| \leq C_\alpha K_{|\beta|} \jxi^{m-|\alpha|}
		\end{equation}
		for all $x,\,\xi \in \R^n$. Here $K_{|\beta|}$ is a weight sequence such that
		\begin{equation}\label{APP:PSEUDO:CALC:CON:EstR:EstK}
		\inf\limits_{p\in \N}\frac{K_{p}}{\jxi^{p}} \leq C e^{- \delta_0 \psi(\jxi)}
		\end{equation}
		for some $\delta_0 > 0$. Furthermore, we suppose that the relation
		\begin{equation}\label{APP:PSEUDO:CALC:CON:EstR:EtaAdditive}
		\psi(\jbl \xi + \zeta \jbr) \leq \psi(\jxi) + \psi(\jbl \zeta \jbr)
		\end{equation}
		holds for all large $\xi,\,\zeta \in \R^n$.
		We assume that the constant $\lambda > 0$ is such that there exists another positive constant $c_0$ such that
		\begin{equation}\label{APP:PSEUDO:CALC:CON:EstR:Lambda}
		\delta_0 - \lambda = c_0 > 0.
		\end{equation}
		Then the remainder $r_N(x,\,\xi)$ given by \eqref{APP:PSEUDO:CALC:CON:eq3} satisfies the estimate
		\begin{equation}\label{APP:PSEUDO:CALC:CON:EstR:EstR}
		\big|\partial_\xi^\alpha D_x^\beta r_N(x,\,\xi)\big| \leq C_{\alpha,\,\beta,\,N}  \lambda^{N}  \jxi^{m-|\alpha| - N} \psi(\jxi)^{N}
		\end{equation}
		for $(x,\,\xi) \in \R^n\times\R^n$ and $\alpha,\beta \in \N^n$.
	\end{Proposition}
	\begin{proof}
		Our proof essentially follows the strategy presented in \cite{Kajitani.1983, Kajitani.1989}.
		We have
		\begin{equation*}
			\begin{aligned}
				|\partial_\xi^\alpha D_x^\beta r_N(x,\,\xi)|
				=
				\bigg|\frac{N}{(2\pi)^n} \sum\limits_{\substack{|\gamma| = N\\\alpha^\prime + \alpha^{\prime\prime} = \alpha}}\binom{\alpha}{\alpha^\prime}
				\Osii \int\limits_0^1 \frac{(1-\vartheta)^{N-1}}{\gamma !}\\
				\times \widetilde F_{\alpha^{\prime\prime},\,\gamma,\,\beta}(x,\,y,\,\xi,\,\zeta,\,\vartheta) G_{\alpha^\prime,\,\gamma}(\xi,\,\zeta) \rmd\vartheta \rmd y \rmd\zeta\bigg|,
			\end{aligned}
		\end{equation*}
		where
		\begin{align*}
			\widetilde F_{\alpha^{\prime\prime},\,\gamma,\,\beta}(x,\,y,\,\xi,\,\zeta,\,\vartheta) &= e^{-\I y \cdot \zeta} \partial_\xi^{\alpha^{\prime\prime}} D_x^{\gamma + \beta} a(x+\vartheta y,\,\xi), \text{ and }\\
			G_{\alpha^\prime,\,\gamma}(\xi,\,\zeta) &= \partial_\xi^{\alpha^\prime}\Big( e^{-\lambda \psi(\jxi)} \partial_\nu^{\gamma}\big( e^{\lambda \psi(\jbl \nu \jbr)}\big|_{\nu = \xi + \zeta}\big)\Big).
		\end{align*}
		Considering $F_{\alpha^{\prime\prime},\,\gamma,\,\beta}= F_{\alpha^{\prime\prime},\,\gamma,\,\beta}(x,\,\xi,\,\zeta)$ with
		\begin{align*}
			F_{\alpha^{\prime\prime},\,\gamma,\,\beta}=  \Osi \int\limits_0^1 (1-\vartheta)^{N-1}  \widetilde F_{\alpha^{\prime\prime},\,\gamma,\,\beta}(x,\,y,\,\xi,\,\zeta,\,\vartheta)  \rmd\vartheta \rmd y,
		\end{align*}
		and obtain for $|\zeta| \geq 1$ that
		\begin{align*}
			|\zeta^\kappa F_{\alpha^{\prime\prime},\,\gamma,\,\beta}|
			&=
			\bigg|\Osi \int\limits_0^1 (1-\vartheta)^{N-1} \zeta^\kappa e^{-\I y \cdot \zeta} \partial_\xi^{\alpha^{\prime\prime}} D_x^{\gamma + \beta} a(x+\vartheta y,\,\xi)  \rmd\vartheta \rmd y\bigg|\\
			&=
			\bigg|\Osi \int\limits_0^1 (1-\vartheta)^{N-1}  e^{-\I y \cdot \zeta} \partial_\xi^{\alpha^{\prime\prime}} D_y^\kappa D_x^{\gamma + \beta} a(x+\vartheta y,\,\xi)  \rmd\vartheta \rmd y\bigg|\\
			&\leq
			C_{N,\,\alpha^{\prime\prime}} K_{|\gamma| + |\beta| + |\kappa|} \jxi^{m-|\alpha^{\prime\prime}|},
		\end{align*}
		where we used $ \zeta^\kappa e^{-\I y \cdot \zeta} = (-D_y)^\kappa e^{-\I y \cdot \zeta}$ and integrated by parts. For $|\zeta| \geq 1$ we know that $\jbl \zeta \jbr \leq \sqrt{2}|\zeta|$. Hence, it is clear that
		\begin{equation*}
			\begin{aligned}
				|F_{\alpha^{\prime\prime},\,\gamma,\,\beta}| &\leq C_{N,\,\alpha^{\prime\prime}} K_{|\gamma| + |\beta| + |\kappa|} \jxi^{m-|\alpha^{\prime\prime}|} |\zeta|^{-|\kappa|}  \\&\leq \sqrt{2}^{|\kappa|} C_{N,\,\alpha^{\prime\prime}} K_{|\gamma| + |\beta| + |\kappa|} \jxi^{m-|\alpha^{\prime\prime}|} \jbl\zeta\jbr^{-|\kappa|}.
			\end{aligned}
		\end{equation*}
		In the case $|\zeta| < 1$ we have $\jbl\zeta\jbr^{|\kappa|} < \sqrt{2}^{|\kappa|}$. This allows us to conclude
		\begin{align*}
			|F_{\alpha^{\prime\prime},\,\gamma,\,\beta}|
			&\leq
			\Osi \int\limits_0^1 (1-\vartheta)^{N-1} \big|\partial_\xi^{\alpha^{\prime\prime}} D_x^{\gamma + \beta} a(x+\vartheta y,\,\xi)\big|  \rmd\vartheta \rmd y\\
			&\leq
			C_{N,\,\alpha^{\prime\prime}} K_{|\gamma| + |\beta|} \jxi^{m-|\alpha^{\prime\prime}|}\\
			&\leq
			\sqrt{2}^{|\kappa|} C_{N,\,\alpha^{\prime\prime}} K_{|\gamma| + |\beta| + |\kappa|} \jxi^{m-|\alpha^{\prime\prime}|} \jbl\zeta\jbr^{-|\kappa|}.
		\end{align*}
		We combine the estimates of both cases and obtain
		\begin{align}
			\nonumber
			|F_{\alpha^{\prime\prime},\,\gamma,\,\beta}(x,\,\xi,\,\zeta)|
			&\leq
			\sqrt{2}^{|\kappa|} C_{N,\,\alpha^{\prime\prime}} K_{|\gamma| + |\beta| + |\kappa|} \jxi^{m-|\alpha^{\prime\prime}|} \jbl\zeta\jbr^{-|\kappa|}\\
			&\leq
			\sqrt{2}^{|\kappa|} C_{N,\,\alpha^{\prime\prime},\,\gamma,\,\beta} K_{|\kappa|} \jxi^{m-|\alpha^{\prime\prime}|} \jbl\zeta\jbr^{-|\kappa|} \label{APP:PSEUDO:CALC:CON:EstR:EstF}
		\end{align}
		for all $x,\,\xi,\,\zeta \in \R^n$.
		
		Next, we consider
		\begin{equation*}
			G_{\alpha^\prime,\,\gamma}(\xi,\,\zeta) = \partial_\xi^{\alpha^\prime}\Big( e^{-\lambda \psi(\jxi)} \partial_\nu^{\gamma}\big( e^{\lambda \psi(\jbl \nu \jbr)}\big|_{\nu = \xi + \zeta}\big)\Big),
		\end{equation*}
		and find that
		\begin{align*}
			\partial_\nu^{\gamma}\big( e^{\lambda \psi(\jbl \nu \jbr)}\big|_{\nu = \xi + \zeta}\big) &=  e^{\lambda \psi(\jbl \xi + \zeta \jbr)} Q_{1,\,\gamma}(\xi + \zeta),\\
			\partial_\xi^{\alpha^\prime}\Big( e^{-\lambda \psi(\jxi)} \partial_\nu^{\gamma }\big( e^{\lambda \psi(\jbl \nu \jbr)}\big|_{\nu = \xi + \zeta}\big)\Big) &= e^{\lambda(\psi(\jbl \xi + \zeta \jbr)- \psi(\jxi))} Q_{2,\,\gamma,\,\alpha^\prime}(\xi+\zeta,\,\xi),
		\end{align*}
		where
		\begin{align}
			\label{APP:PSEUDO:CALC:CON:EstR:EstQ1}
			|Q_{1,\,\gamma}(\xi + \zeta)| &\leq C_{\gamma} \lambda^{|\gamma|} \psi^{\prime}(\jbl \xi + \zeta \jbr)^{|\gamma|},\\
			\nonumber
			|Q_{2,\,\gamma,\,\alpha^\prime}(\xi+\zeta,\,\xi)| &\leq C_{\gamma,\,\alpha^{\prime}} |Q_{1,\,\gamma}(\xi + \zeta)| \lambda^{|\alpha^{\prime}|} \big|\psi^{\prime}(\jbl \xi + \zeta \jbr) - \psi^\prime(\jxi)\big|^{|\alpha^{\prime}|}\\
			&\leq\label{APP:PSEUDO:CALC:CON:EstR:EstQ2}
			C_{\gamma,\,\alpha^{\prime}}  \lambda^{|\gamma|+ |\alpha^{\prime}|} \psi^{\prime}(\jbl \xi + \zeta \jbr)^{|\gamma|} \big|\psi^{\prime}(\jbl \xi + \zeta \jbr) - \psi^\prime(\jxi)\big|^{|\alpha^{\prime}|}.
		\end{align}
		For some $\widetilde \vartheta \in (0,\,1)$, we have the estimate
		\begin{align}
			\nonumber
			|\psi^{\prime}(\jbl \xi + \zeta \jbr) - \psi^\prime(\jxi)|^{|\alpha^{\prime}|}
			&=
			|\psi^{\prime\prime}(\jbl \xi + \zeta \jbr - \widetilde \vartheta \jxi)|^{|\alpha^\prime|} |\jbl \xi + \zeta \jbr - \jxi|^{|\alpha^\prime|}\\
			\nonumber
			&\leq
			C \bigg|\frac{\psi(\jbl \xi + \zeta\jbr - \vartheta \jxi)}{|\jbl \xi + \zeta \jbr - \widetilde \vartheta \jxi|^2}\bigg|^{|\alpha^\prime|} \jbl  \zeta \jbr^{|\alpha^\prime|}\\
			&\leq
			C \jxi^{-|\alpha^\prime|}\jbl  \zeta \jbr^{|\alpha^\prime|}, \label{APP:PSEUDO:CALC:CON:EstR:EstPsi}
		\end{align}
		where we applied assumption \eqref{APP:PSEUDO:CALC:CON:EstEta}. Combining \eqref{APP:PSEUDO:CALC:CON:EstR:EstQ1}, \eqref{APP:PSEUDO:CALC:CON:EstR:EstQ2} and \eqref{APP:PSEUDO:CALC:CON:EstR:EstPsi}, we obtain
		\begin{align*}
			|G_{\alpha^\prime,\,\gamma}(\xi,\,\zeta)|
			&\leq
				\begin{aligned}[t]
					C_{\gamma,\,\alpha^{\prime}}  \lambda^{|\gamma|+ |\alpha^{\prime}|} e^{\lambda \big(\psi(\jbl \xi + \zeta \jbr) - \psi(\jxi)\big)} &\psi^{\prime}(\jbl \xi + \zeta \jbr)^{|\gamma|}\\ &\times\big|\psi^{\prime}(\jbl \xi + \zeta \jbr) - \psi^\prime(\jxi)\big|^{|\alpha^{\prime}|}
				\end{aligned}\\
			&\leq
			C_{\gamma,\,\alpha^{\prime}}  \lambda^{|\gamma|+ |\alpha^{\prime}|} e^{\lambda \big(\psi(\jbl \xi + \zeta \jbr) - \psi(\jxi)\big)} \frac{\psi(\jbl \xi + \zeta \jbr)^{|\gamma|}}{\jbl \xi + \zeta \jbr^{|\gamma|}} \jxi^{-|\alpha^\prime|} \jbl\zeta\jbr^{|\alpha^\prime|}\\
			&\leq
			C_{\gamma,\,\alpha^{\prime}}  \lambda^{|\gamma|+ |\alpha^{\prime}|} e^{\lambda \big(\psi(\jbl \xi + \zeta \jbr) - \psi(\jxi)\big)} \frac{\jbl\zeta\jbr^{|\alpha^\prime|}\psi(\jbl \xi + \zeta \jbr)^{|\gamma|}}{\jxi^{|\alpha^\prime|+ |\gamma|}}.
		\end{align*}
		In view of assumption \eqref{APP:PSEUDO:CALC:CON:EstR:EtaAdditive}, i.e. $\psi(\jbl \xi + \zeta \jbr) \leq \psi(\jxi) + \psi(\jbl \zeta \jbr)$, we conclude that
		\begin{equation}
		|G_{\alpha^\prime,\,\gamma}(\xi,\,\zeta)|
		\leq
		C_{\gamma,\,\alpha^\prime} \lambda^{|\gamma|+|\alpha^{\prime}| } e^{\lambda \psi(\jbl\zeta \jbr) } \frac{\jbl\zeta\jbr^{|\alpha^\prime|}(\psi(\jxi) + \psi(\jbl \zeta \jbr))^{|\gamma|}}{\jxi^{|\alpha^\prime|+ |\gamma|}}.
		\label{APP:PSEUDO:CALC:CON:EstR:EstG}
		\end{equation}
		Combining \eqref{APP:PSEUDO:CALC:CON:EstR:EstF} and \eqref{APP:PSEUDO:CALC:CON:EstR:EstG} implies
		\begin{equation*}
			\begin{aligned}
				|\partial_\xi^\alpha D_x^\beta r_N(x,\,\xi)|
				&\leq
				\frac{N}{(2\pi)^n N!} \sum\limits_{\substack{|\gamma| = N\\\alpha^\prime + \alpha^{\prime\prime} = \alpha\\\delta^\prime + \delta^{\prime\prime} = \alpha^\prime}}
				\binom{\alpha}{\alpha^\prime}\binom{\alpha^\prime}{\delta^\prime}
				C_{N,\,\alpha,\,\beta}\sqrt{2}^{|\kappa|}\Osi e^{\lambda \psi(\jbl\zeta \jbr)}\\ &\times\lambda^{|\gamma|} K_{|\kappa|} \jbl\zeta\jbr^{-(|\kappa|-|\alpha^\prime|)} (\psi(\jxi) + \psi(\jbl \zeta \jbr))^{|\gamma|}
				\jxi^{m-|\alpha| - |\gamma|} \rmd\zeta.
			\end{aligned}
		\end{equation*}
		We use $K_{|\kappa|} \leq C_\alpha K_{|\kappa|-|\alpha^\prime|}$ and apply assumption \eqref{APP:PSEUDO:CALC:CON:EstR:EstK} to choose $\kappa$ such that
		\begin{equation*}
			K_{|\kappa|-|\alpha^\prime|} \jbl\zeta\jbr^{-(|\kappa|-|\alpha^\prime|)} \leq C e^{-\delta_0 \psi(\jbl \zeta \jbr)}.
		\end{equation*}
		This yields
		\begin{equation*}
			\begin{aligned}
				|\partial_\xi^\alpha D_x^\beta r_N(x,\,\xi)|
				\leq
				\frac{N}{(2\pi)^n N!} \sum\limits_{|\gamma| = N}
				C_{N,\,\alpha,\,\beta}\sqrt{2}^{|\kappa|} \lambda^{|\gamma|}
				\Osi e^{-(\delta_0-\lambda)\psi(\jbl\zeta \jbr)}\\\times (\psi(\jxi) + \psi(\jbl \zeta \jbr))^{|\gamma|}\jxi^{m-|\alpha| - |\gamma|} \rmd\zeta.
			\end{aligned}
		\end{equation*}
		By assumption \eqref{APP:PSEUDO:CALC:CON:EstR:Lambda} there exists a positive constant $c_0$ such that $\delta_0 - \lambda = c_0 > 0$. We conclude that
		\begin{align*}
			|\partial_\xi^\alpha D_x^\beta r_N(x,\,\xi)|
			&\leq
			\begin{aligned}[t]
				C_{\alpha,\,\beta,\,N} \sum\limits_{|\gamma| = N} \jxi^{m-|\alpha| - |\gamma|}
				\lambda^{|\gamma|}
				\Osi e^{- c_0 \psi(\jbl\zeta \jbr)} \sqrt{2}^{|\kappa|} 
				\\\times (\psi(\jxi) + \psi(\jbl \zeta \jbr))^{|\gamma|}\rmd\zeta
			\end{aligned}\\
			&\leq
			\begin{aligned}[t]
				C_{\alpha,\,\beta,\,N} \lambda^{N} \jxi^{m-|\alpha| - N}   \sum\limits_{|\gamma| = N}	 
				\Osi e^{- c_0 \psi(\jbl\zeta \jbr)} 
				\sqrt{2}^{|\kappa|}\\\times(\psi(\jxi)^{|\gamma|} + \psi(\jbl \zeta \jbr)^{|\gamma|})  \rmd\zeta
			\end{aligned}\\
			&\leq
			C_{\alpha,\,\beta,\,N}  \lambda^{N}  \jxi^{m-|\alpha| - N} \psi(\jxi)^{N} \Osi e^{- \frac{c_0}{2} \psi(\jbl\zeta \jbr)} \sqrt{2}^{|\kappa|}  \rmd\zeta\\
			&\leq
			C_{\alpha,\,\beta,\,N}  \lambda^{N}  \jxi^{m-|\alpha| - N} \psi(\jxi)^{N}.
		\end{align*}
		This concludes the proof.	
	\end{proof}

	\section{Statement of the Results}\label{RESULTS}
	
	Depending on the modulus of continuity of the coefficients, we expect to have an at most finite loss of derivatives (for strong moduli of continuity) or an infinite loss of derivatives (for weak moduli of continuity). We account for this difference by stating two different theorems, one for strong and one for weak moduli of continuity.
	
	In both cases we consider the Cauchy problem
	\begin{equation}\label{SH:MR:THEOREM:WEAK:CauchyProblem}
	\begin{cases}
	D_t^m u = \sum\limits_{j=0}^{m-1} A_{m-j}(t,\,x,\,D_x) D_t^j u  + f(t,x),\\
	D_t^{k-1} u(0,x) = g_k(x), \quad (t,x) \in [0, T]\times \R^n,\,k = 1,\,\ldots,\,m,
	\end{cases}
	\end{equation}
	where
	\begin{equation*}
		A_{m-j}(t,\,x,\,D_x) = \sum\limits_{|\gamma| + j = m} a_{m-j, \gamma}(t,\,x) D_x^\gamma + \sum\limits_{|\gamma| + j \leq m-1} a_{m-j,\gamma}(t,\,x) D_x^\gamma.
	\end{equation*}
	We assume in both cases the following conditions:
	\begin{enumerate}[label = (SH\arabic*), align = left, leftmargin=*]
		\item \label{RESULTS:Both:1} \label{SH:MR:THEOREM:STRONG:ASSUME:StrictHyp} The Cauchy problem is strictly hyperbolic, i.e. that the characteristic roots $\tau_k(t,\,x,\,\xi),~k = 1,\,\ldots,\,m,$ of the principal part
		\begin{equation*}
			\begin{aligned}
			P_m(t,\,x,\,\tau,\,\xi) &= \tau^m - \sum\limits_{j=0}^{m-1} A_{(m-j)}(t,\,x,\,\xi) \tau^j\\ &= \tau^m - \sum\limits_{j=0}^{m-1} \sum\limits_{|\gamma| + j = m} a_{m-j, \gamma}(t,\,x) \xi^\gamma \tau^j,
			\end{aligned}
		\end{equation*}
		are real when $|\xi| \neq 0$, simple and numbered in such a way that
		\begin{equation*}
			\tau_1(t,\,x,\,\xi) < \tau_2(t,\,x,\,\xi) < \ldots < \tau_m(t,\,x,\,\xi),
		\end{equation*}
		for all $t\in [0,\,T],\,x,\,\xi \in \R^n$.
		\item \label{RESULTS:Both:2} The coefficients $a_{m-j,\,\gamma}(t,\,x)$ of the lower order terms (i.e. $|\gamma| + j \leq m-1$) are continuous in time and the coefficients $a_{m-j,\,\gamma}(t,\,x)$ of the principal part (i.e. $|\gamma| + j = m$) satisfy
		\begin{equation*}
		\big|D_x^\beta a_{m-j,\,\gamma}(t,\,x) - D_x^\beta a_{m-j,\,\gamma}(s,\,x)\big| \leq C K_{|\beta|} \mu(|t-s|),
		\end{equation*}
		for some constant $C > 0$ and a weight sequence $K_p$, all $t,s \in [0,\, T]$, all $\beta\in\N^n$ and fixed $x\in\R^n$, where $\mu$ is a strong or weak modulus of continuity.
		\item \label{RESULTS:Both:3} The modulus of continuity $\mu$ in \ref{RESULTS:Both:2} can be written in the form
		\begin{equation*}
			\mu(s) = s\omega(s^{-1}),
		\end{equation*}
		where $\omega(s)$ is a non-decreasing, smooth function on $[0,\,1]$.
	\end{enumerate}

	\subsection{Result for Strong Moduli of Continuity}\label{SH:MR:THEOREM:STRONG}
	
	We are able to prove the following well-posedness result if we assume \ref{RESULTS:Both:1}, \ref{RESULTS:Both:2} and \ref{RESULTS:Both:3} as well as the following conditions:
	\begin{enumerate}[label = (SH\arabic*-S), align = left, leftmargin=*]
		\item \label{SH:MR:THEOREM:STRONG:ASSUME:Coeff:Time} The modulus of continuity $\mu$ in \ref{RESULTS:Both:2} is strong and the weight sequence $K_p$ is arbitrary.
		\item \label{SH:MR:THEOREM:STRONG:ASSUME:Coeff:Space}
		All the coefficients $a_{m-j,\,\gamma}$ belong to $C\big([0,\,T];\,\B^\infty\big)$.
		\item \label{SH:MR:THEOREM:STRONG:ASSUME:DATA} The initial data $g_k$ belongs to $H^{\nu+m-k}$ for $k=1,\,\cdots,\,m$.
		\item \label{SH:MR:THEOREM:STRONG:ASSUME:Inhomogeneity} The right-hand side $f=f(t,x) \in C([0,\,T];\,H^\nu)$.
		\item \label{SH:MR:THEOREM:STRONG:ASSUME:Eta} The function $\omega$ is smooth and satisfies
		\begin{equation}\label{SH:MR:THEOREM:STRONG:ASSUME:Eta:diff}
		\bigg|\frac{\rmd^k}{\rmd s^k}\omega (s) \bigg| \leq C_k s^{-k} \omega(s)
		\end{equation}
		for all $k \in \N$ and large $s \in \R^+$.
	\end{enumerate}
	
	\begin{Theorem}\label{SH:MR:THEOREM:STRONG:MainTheorem}
		Consider the \hyperref[SH:MR:THEOREM:WEAK:CauchyProblem]{Cauchy problem \eqref{SH:MR:THEOREM:WEAK:CauchyProblem}}. Under the above assumptions, there is a $\kappa> 0$ such that for every $s \in \R$ there exists a unique global (in time) solution
		\begin{equation*}
			u \in \bigcap\limits_{j = 0}^{m-1} C^{m-1-j}\big([0,T];\, H^{\nu+j}_{\omega,\,-\kappa T}\big).
		\end{equation*}
		The solution satisfies the a-priori estimate
		\begin{equation*}
			\begin{aligned}[t]
				& \sum\limits_{j = 0}^{m-1} \big\|\jbl D_x \jbr^{\nu+m-1-j} e^{- \kappa t \omega(\jbl D_x \jbr)} \partial_t^j u(t,\,\cdot)\big\|^2_{L^2} \\ & \quad \leq C \Big(\sum\limits_{k = 1}^{m} \big\|\jbl D_x \jbr^{\nu+m-k} g_k(0,\,\cdot)\big\|^2_{L^2} + \int\limits_0^t \big\|\jbl D_x \jbr^\nu e^{- \kappa z \omega(\jbl D_x \jbr)} f(z,\,\cdot)\big\|^2_{L^2} \rmd z\Big)
			\end{aligned}
		\end{equation*}
		for $0 \leq t \leq T$ and some $C = C_s > 0$.
	\end{Theorem}
	We note that the spaces $H^{\nu}_{\omega,\,\-\kappa T}$ are imbedded into Sobolev spaces, since we are dealing with strong moduli of continuity, i.e. $\omega(s) = O(\log(s))$. This means, that for $\omega(s) = 1$, i.e. Lipschitz-continuous coefficients, we have no loss of derivatives. For coefficients that are Log-Lipschitz-continuous in time, we have well-posedness in $H^\nu_{\log,\,-\kappa T} = H^{\nu - \kappa T}$ with at most finite loss of derivatives. In between both cases, the loss of derivatives is arbitrarily small.

	\subsection{Result for Weak Moduli of Continuity}\label{SH:MR:THEOREM:WEAK}

	We are able to prove the following well-posedness result if we assume \ref{RESULTS:Both:1}, \ref{RESULTS:Both:2} and \ref{RESULTS:Both:3} as well as the following conditions:
	\begin{enumerate}[label = (SH\arabic*-W), align = left, leftmargin=*]
		\item \label{SH:MR:THEOREM:WEAK:ASSUME:Coeff:Time} The modulus of continuity $\mu$ in \ref{RESULTS:Both:2} is weak.
		\item \label{SH:MR:THEOREM:WEAK:ASSUME:Coeff:Space}
		All the coefficients $a_{m-j,\,\gamma}$ belong to $C\big([0,\,T];\,\B_K^\infty\big)$.
		\item \label{SH:MR:THEOREM:WEAK:ASSUME:DATA} The initial data $g_k$ belongs to $H^{\nu+m-k}_{\eta,\,\delta_1}$ for $k=1,\,\cdots,\,m$.
		\item \label{SH:MR:THEOREM:WEAK:ASSUME:Inhomogeneity} The right-hand side $f=f(t,x) \in C([0,\,T];\,H^\nu_{\eta,\,\delta_2})$.
				
		\item \label{SH:MR:THEOREM:WEAK:ASSUME:ConEtaK} The weight function $\eta$ and the sequence of constants $K_{p}$ satisfy the relation
		\begin{equation*}
			\inf\limits_{p\in\N}\frac{K_p}{\jxi^{p}} \leq C e^{-\delta_0 \eta(\jbl\xi\jbr)}
		\end{equation*}
		for large $|\xi|$ and some $\delta_0 > 0$.
		\item \label{SH:MR:THEOREM:WEAK:ASSUME:Eta} The functions $\eta$ and $\omega$ are smooth and satisfy
		\begin{equation}\label{SH:MR:THEOREM:WEAK:ASSUME:Eta:diff}
		\bigg|\frac{\rmd^k}{\rmd s^k}\eta (s) \bigg| \leq C_k s^{-k} \eta(s),\qquad
		\bigg|\frac{\rmd^k}{\rmd s^k}\omega (s) \bigg| \leq C_k s^{-k} \omega(s),
		\end{equation}
		for all $k \in \N$ and large $s \in \R^+$ and
		\begin{equation}\label{SH:MR:THEOREM:WEAK:ASSUME:Eta:subadd}
		\eta(\jbl \xi + \zeta \jbr) \leq \eta(\jxi) + \eta(\jbl \zeta \jbr),\qquad
		\omega(\jbl \xi + \zeta \jbr)\leq \omega(\jxi) + \omega(\jbl \zeta \jbr)
		\end{equation}
		for all large $\xi,\,\zeta \in \R^n$.
	\end{enumerate}
	\begin{Theorem}\label{SH:MR:THEOREM:WEAK:MainTheorem}
		Consider the \hyperref[SH:MR:THEOREM:WEAK:CauchyProblem]{Cauchy problem \eqref{SH:BM:CauchyProblem}}. Under the above assumptions, there is a unique global (in time) solution
		\begin{equation*}
			u \in \bigcap\limits_{j = 0}^{m-1} C^{m-1-j}\big([0,T];\, H^{\nu+j}_{\eta,\,\delta}\big),
		\end{equation*}
		where $\delta < \min\{\delta_0,\,\delta_1,\,\delta_2\}$, provided that,
		\begin{equation}\label{SH:MR:THEOREM:WEAK:MainTheorem:Equ}
		\mu\left(\frac{1}{\jbl\xi\jbr}\right) \jbl\xi\jbr = \omega(\jbl\xi\jbr) = o(\eta(\jbl\xi\jbr)).
		\end{equation}		
		More specifically, for any $t_0 \in [0,\,T]$, any sufficiently large $\kappa$ and for any $T^\ast \in [0,\,T-t_0]$ such that $\kappa T^\ast$ is sufficiently small, we have the a-priori estimate
		\begin{equation*}
			\begin{aligned}[t]
				& \sum\limits_{j = 0}^{m-1} \big\|\jbl D_x \jbr^{\nu + m-1 -j} e^{ \kappa(T^\ast + t_0 - t) \omega(\jbl D_x \jbr)} \partial_t^j u(t,\,\cdot)\big\|^2_{L^2} \\ & \qquad \leq C \Big(\sum\limits_{k = 1}^{m} \big\|\jbl D_x \jbr^{\nu + m -k} e^{ \kappa (T^\ast+t_0) \omega(\jbl D_x \jbr)} g_k(t_0,\,\cdot)\big\|^2_{L^2}\\
				&\qquad\qquad + \int\limits_{t_0}^t \big\|\jbl D_x \jbr^\nu e^{\kappa(T^\ast + t_0 - z)  \omega(\jbl D_x \jbr)} f(z,\,\cdot)\big\|^2_{L^2} \rmd z\Big),
			\end{aligned}
		\end{equation*}
		for $t_0 \leq t \leq t_0 + T^\ast$.
	\end{Theorem}	
	We note that the spaces $H^{\nu+j}_{\eta,\,\delta}$ are spaces of ultra-differentiable functions, since we are dealing with weak moduli of continuity, i.e. $\omega(s) = o(\log(s))$. In general, the loss of derivatives that occurs is infinite.	
		
	\subsection{Examples and Remarks}
		
		Let us begin with some examples of strong moduli of continuity.
		
		\begin{Example}[Lipschitz-coefficients]\label{SH:MR:EXAMPLES:Lip}
			Let the coefficients be Lipschitz continuous in time, i.e. $\mu(s) = s$ and $\omega(s) = 1$, and $\B^\infty$ in space.
			The initial data $g_k$ and right-hand side $f$ are chosen such that
			\begin{equation*}
				g_k \in H^{\nu+m-k},\, k = 1,\,\ldots,\,m,\quad f \in C\big([0,\,T];\,H^\nu\big).
			\end{equation*}
			All assumptions are satisfied and we a have global (in time) solution
			\begin{equation*}
			u = u(t,\,x) \in \bigcap\limits_{j = 0}^{m-1} C^{m-1-j}\big([0,T];\, H^{\nu+j}\big),
			\end{equation*} i.e. we have Sobolev-well-posedness without a loss of derivatives. This result is already well-known (see e.g. \cite{Agliardi.2004, Cicognani.1999, Cicognani.2006}, \cite[Chapter 9]{Hormander.1963} and \cite[Chapter 6]{Mizohata.1973}).
		\end{Example}
		
		\begin{Example}[Log-Lip-coefficients]\label{SH:MR:EXAMPLES:LogLip}
			Let the coefficients be Log-Lip-continuous in time, i.e. $\mu(s) = s\big(\log(s^{-1}) +1\big)$ and $\omega(s) = \log(s) +1$, and  $\B^\infty$ in space.
			The initial data $g_k$ and right-hand side $f$ are chosen such that
			\begin{equation*}
				g_k \in H^{\nu+m-k},\, k = 1,\,\ldots,\,m,\quad f \in C\big([0,\,T];\,H^{\nu}\big).
			\end{equation*}
			In this case, we have a global (in time) solution
			\begin{equation*}
			u = u(t,\,x) \in \bigcap\limits_{j = 0}^{m-1} C^{m-1-j}\big([0,T];\,H^{\nu+j-\kappa T}\big),
			\end{equation*} where $\kappa > 0$ is sufficiently large.  We conclude that we have an at most finite loss of derivatives, which is also a well-known result (see e.g. \cite{Agliardi.2004, Cicognani.1999, Cicognani.2006, Colombini.1995}).
		\end{Example}

		Next, we consider weak moduli of continuity. Looking at Theorem~\ref{SH:MR:THEOREM:WEAK:MainTheorem}, assumptions \ref{SH:MR:THEOREM:WEAK:ASSUME:ConEtaK} and \ref{SH:MR:THEOREM:WEAK:ASSUME:Eta} may not be as intuitive as the other ones.
		
		Assumption \ref{SH:MR:THEOREM:WEAK:ASSUME:ConEtaK} describes the connection between the weight function $\eta$ of the solution space and the behavior of the coefficients with respect to the spatial variables. In a way, we may interpret this condition as a multiplication condition in the sense that the regularity of the coefficients in $x$ has to be such that the product of coefficients and the solution stays in the solution space. This means that the weight sequence $K_p$ and the weight function $\eta$ have to be compatible in a certain sense. One way to ensure that they are compatible is to choose them such that the function space of all functions $f \in C^\infty(\R^n)$ with
		\begin{equation*}
			\sup\limits_{x\,\in \R^n}|D_x^\alpha f(x)| \leq C K_{|\alpha|},
		\end{equation*}
		and the function space of all functions $f \in L^2(\R^n)$ with
		\begin{equation*}
			e^{\eta(\langle D_x \rangle)} f \in L^2(\R^n)
		\end{equation*}
		coincide. For results concerning the conditions on $\eta$ and $K_p$ under which both spaces coincide, we refer the reader to \cite{Bonet.2007, Pascu.2010, Reich.2016}.
		
		Assumption \ref{SH:MR:THEOREM:WEAK:ASSUME:Eta} provides some relations that are used in the pseudodifferential calculus. Condition \eqref{SH:MR:THEOREM:WEAK:ASSUME:Eta:diff} for $\eta$ and $\omega$ is not really a restriction. If $\eta$ or $\omega$ happen to be not smooth, we can define equivalent weight functions, that are smooth and satisfy \eqref{SH:MR:THEOREM:WEAK:ASSUME:Eta:diff}.
		
		The difficulty of checking whether condition \eqref{SH:MR:THEOREM:WEAK:ASSUME:Eta:subadd}, is satisfied, certainly depends on the choice of $\eta$ and $\omega$. However, in some cases it may be easier to verify that $\eta$  and $\omega$ belong to a certain class of weights, for which \eqref{SH:MR:THEOREM:WEAK:ASSUME:Eta:subadd} is satisfied. An example of such a class is introduced in Definition~3.7 in \cite{Reich.2016}.
		
		In the following, we compute some examples for weak moduli of continuity. In each example, we first choose a certain modulus of continuity $\mu$ to describe the regularity of the coefficients in time. Depending on this modulus of continuity, we look for a suitable weight function $\eta$ which satisfies \eqref{SH:MR:THEOREM:WEAK:MainTheorem:Equ}.
		Having chosen $\eta$, we specify the regularity of the coefficients in space by choosing a sequence of constants $K_p$ such that \ref{SH:MR:THEOREM:WEAK:ASSUME:ConEtaK} is satisfied.
		
		\begin{Example}[Log-Log$^{[m]}$-Lip coefficients]\label{SH:MR:EXAMPLES:LogLogmLip}
			Let the coefficients be Log-Log$^{[m]}$-Lip continuous in time, i.e. $\mu(s) =  s \left(\log\left(\frac{1}{s}\right) + 1\right)\log^{[m]}\left(\frac{1}{s}\right)$ and $\omega(s) = (\log(s) + 1)\log^{[m]}(s)$, $m \geq 2$. We choose
			\begin{equation*}
				\eta(s) =  \log(s)\big(\log^{[m]}(s)\big)^{1+\ve} + c_m,
			\end{equation*}
			where $\ve >0$ is arbitrarily small and $c_m>0$ is such that $\eta(s) \geq 1$ for all $s \geq 1$.
			Finding a suitable weight sequence $K_p$ is not obvious. In the literature on weight sequences and weight functions (e.g. \cite{Bonet.2007, Komatsu.1977, Pilipovich.2015, Pilipovich.2016}), for a given weight function $M(t)$ one can compute the associated weight sequence $\{M_p\}_p$ by considering
			\begin{equation*}
				M_p = \sup\limits_{t > 0} \frac{t^p}{e^{M(t)}}.
			\end{equation*}
			Following this approach for our weight function $\eta$ runs into the difficulty of actually computing the above supremum. However, it is possible to compute the associated sequence for $\omega$, which is $M_{p,\,\omega} = (\exp^{[m]}(p))^{(p-1)}e^{-p}$.
			From this, we can get a possible sequence $K_p$ by decreasing the growth of $M_{p,\,\omega}$. Setting $K_p = (\exp^{[m]}(p))^{(p-1)}e^{-p-\widetilde \ve}$, for some $\widetilde \ve > 0$, yields a possible weight sequence for $\eta$, in the sense that, for fixed $\widetilde \ve > 0$ there is a $\ve > 0$ such that \ref{SH:MR:THEOREM:WEAK:ASSUME:ConEtaK} is satisfied. However it is not clear how $\widetilde \ve$ is related to $\ve$ in general. Furthermore, it is not clear whether the set of functions defined by $K_p$ forms a proper function space.			
			
			Another way to obtain a possible weight sequence $K_p$ (which actually defines a function space) is to guess a weight sequence $\{M_p\}_p$ and to compute the associated weight function $M(t)$ by considering
			\begin{equation*}
				M(t) = \sup\limits_{p \in N} \log\Big(\frac{|t|^p}{M_p}\Big)\; \text{ if } t \neq 0 \text{ and } M(0) = 0.
			\end{equation*}
			Then we compare $M(t)$ and $\eta$. If $M(t)$ grows faster than $\eta$, then $M_p$ is a potential candidate for $K_p$. Of course, following this procedure, we cannot ensure that the sequence $K_p$ is optimal.
			
			In our case, we consider the sequence $\{M_p\}_p = \{p^{p^2}\}_p$, which defines a space of ultradifferentiable functions (see \cite{Pilipovich.2015, Pilipovich.2016}).
			We compute the associated function $M(\xi)$ and obtain
			\begin{equation*}
				M(\xi) = \frac{\log(\jxi)}{2} e^{W\big(\frac{\log(\jxi) \sqrt{e}}{2}\big) -\frac{1}{2}},
			\end{equation*}
			where $W$ denotes the Lambert $W$ function (also called product logarithm).
			For sufficiently large $\xi$ we find that $M(\xi) < \eta(\jxi)$ and, therefore, we obtain the inequality
			\begin{equation*}
				\inf\limits_{p \in \N} \frac{p^{p^2}}{\jxi^p} \leq C e^{-  M(\xi)} \leq C e^{-\delta_0 \eta(\jxi)},
			\end{equation*}
			for some $\delta_0 > 0$.
			
			Either way, condition \ref{SH:MR:THEOREM:WEAK:ASSUME:ConEtaK} is satisfied if the coefficients $a_{m-j,\,\gamma}$ are Log-Log$^{[m]}$-Lip-continuous in time, belong to $\B^\infty$ in space and satisfy
			\begin{equation*}
				|D_x^\beta a_{m-j,\,\gamma}(t,\,x)| \leq C K_{|\beta|},
			\end{equation*}
			uniformly in $x$, for fixed $t$.
			
			The initial data $g_k$  and right-hand side $f$ are chosen such that
			\begin{equation*}
				g_k \in H^{\nu+m-k}_{\eta,\, \delta_1},\, k = 1,\,\ldots,\,m,\quad f \in C\big([0,\,T];\,H^{\nu}_{\eta,\,\delta_2}\big).
			\end{equation*}
			
			Lastly, we check that assumption \ref{SH:MR:THEOREM:WEAK:ASSUME:Eta} is satisfied. Clearly,
			\begin{equation*}
				\Big|\frac{\rmd^k}{\rmd s^k}\eta (s) \Big| \leq C_k s^{-k} \eta(s),\qquad
				\Big|\frac{\rmd^k}{\rmd s^k}\omega (s) \Big| \leq C_k s^{-k} \omega(s)
			\end{equation*}
			for all $k \in \N$ and large $s \in \R$.
			
			Proving
			\begin{equation}\label{SH:MR:EXAMPLES:LogLogmLip:Assume:Eta}
			\eta(\jbl \xi + \zeta \jbr) \leq \eta(\jxi) + \eta(\jbl \zeta \jbr),\qquad
			\omega\jbl \xi + \zeta \jbr)\leq \omega(\jxi) + \omega(\jbl \zeta \jbr),
			\end{equation}
			for large $\jxi$ and $\jbl \zeta \jbr$, is not so obvious.
			We use that $\eta$ and $\omega$ belong to the set $\W(\R)$ which was introduced by Reich~\cite{Reich.2016}. In \cite{Reich.2016} the author proves that all functions in  $\W(\R)$ satisfy an even stronger condition than \eqref{SH:MR:EXAMPLES:LogLogmLip:Assume:Eta}.		
		
			Thus, all assumptions are satisfied and we have a global (in time) solution $u = u(t,\,x) \in \bigcap\limits_{j = 0}^{m-1} C^{m-1-j}\big([0,T];\,H^{\nu+j}_{\eta,\,\delta}\big)$ with $\delta < \min\{\delta_0,\,\delta_1,\,\delta_2\}$. We expect an infinite loss of derivatives since $\log(s) = o(\eta(s))$.
		\end{Example}
		
		\begin{Example}[H{\"o}lder-coefficients]\label{SH:MR:EXAMPLES:Hoelder}
			Let the coefficients be H{\"o}lder-continuous in time, i.e. $\mu(s) =  s^\alpha$ and $\omega(s) = s^{1-\alpha}$, $\alpha\in(0,\,1)$. We choose
			\begin{equation*}
				\eta(s) =  s^\kappa,
			\end{equation*}
			where $\kappa > 1- \alpha$ is a constant.
			We use the well-known inequality
			\begin{equation*}
				\inf\limits_{p \in \N} (p!)^\frac{1}{\kappa} (A \jxi^{-1})^p \leq C e^{-\delta_0 \eta(\jxi)},
			\end{equation*}
			where $A$ is a positive constant. This yields that condition \ref{SH:MR:THEOREM:WEAK:ASSUME:ConEtaK} is satisfied if the coefficients $a_{m-j,\,\gamma}$ are H{\"o}lder-continuous in time and belong to the $\B^\infty$ in space and satisfy
			\begin{equation*}
				|D_x^\beta a_{m-j,\,\gamma}(t,\,x)| \leq C (|\beta|!)^{\frac{1}{\kappa}} A^{|\beta|},
			\end{equation*}
		uniformly in $x$, for fixed $t$. This defines the Gevrey space $G^{\frac{1}{\kappa}}(\R^n)$. The initial data $g_k$  and right-hand side $f$ are chosen such that
			\begin{equation*}
				g_k  \in H^{\nu+m-k}_{\eta,\, \delta_1},\, k = 1,\,\ldots,\,m,\quad f \in C\big([0,\,T];\,H^{\nu}_{\eta,\,\delta_2}\big).
			\end{equation*}
			Lastly, we check that assumption \ref{SH:MR:THEOREM:WEAK:ASSUME:Eta} is satisfied. Clearly,
			\begin{equation*}
				\Big|\frac{\rmd^k}{\rmd s^k}\eta (s) \Big| \leq C_k s^{-k} \eta(s),\quad
				\Big|\frac{\rmd^k}{\rmd s^k}\omega (s) \Big| \leq C_k s^{-k} \omega(s)
			\end{equation*}
			for all $k \in \N$ and large $s \in \R$. Moreover,
			\begin{equation*}
				\omega(\jbl \xi + \zeta \jbr)\leq \omega(\jxi) + \omega(\jbl \zeta \jbr),\quad \text{ and }\quad \eta(\jbl \xi + \zeta \jbr)\leq \eta(\jxi) + \eta(\jbl \zeta \jbr)
			\end{equation*}
			for sufficiently large $\jxi$ and $\jbl \zeta \jbr$.
			
			Thus, all assumptions are satisfied and we have a global (in time) solution $u = u(t,\,x) \in \bigcap\limits_{j = 0}^{m-1} C^{m-1-j}\big([0,T];\,H^{\nu+j}_{\eta,\,\delta}\big)$ with $\delta < \min\{\delta_0,\,\delta_1,\,\delta_2\}$, where the loss of derivatives is infinite, since $\log(s) = o(\eta(s))$.
			
			We remark, that $H^\nu_{\eta,\,\delta}$ is the classical Gevrey space $G^{\frac{1}{\kappa}}$, which means we have Gevrey-well-posedness (with infinite loss of derivatives), if
			$
			\kappa > 1 - \alpha
			$,
			which is a well-known result (cf. e.g. \cite{Agliardi.2004b,Cicognani.1999,Jannelli.1985, Nishitani.1983}).
		\end{Example}
		
		\begin{Example}[Log$^{-\alpha}$-coefficients]\label{SH:MR:EXAMPLES:Log-alpha}
			Let the coefficients be Log$^{-\alpha}$-continuous in time, i.e. $\mu(s) = \left(\log\left(\frac{1}{s}\right) + 1 \right)^{-\alpha}$ and $\omega(s) = s \left(\log(s) + 1 \right)^{-\alpha}$, $\alpha\in(0,\,1)$. We choose
			\begin{equation*}
				\eta(s) =  s \left(\log(s) + 1 \right)^{-\kappa},
			\end{equation*}
			where $0 < \kappa < \alpha$ is a constant.
			In view of Definition~9 and Example~25 in \cite{Bonet.2007}, we find that condition \ref{SH:MR:THEOREM:WEAK:ASSUME:ConEtaK}	is satisfied if we choose
			\begin{equation*}
				K_p = ((p+1)(\log(e+p)))^p.
			\end{equation*}
			The initial data $g_k$  and right-hand side $f$ are chosen such that
			\begin{equation*}
				g_k \in H^{\nu+m-k}_{\eta,\, \delta_1},\, k = 1,\,\ldots,\,m,\quad f \in C\big([0,\,T];\,H^{\nu}_{\eta,\,\delta_2}\big).
			\end{equation*}
			Lastly, we check that assumption \ref{SH:MR:THEOREM:WEAK:ASSUME:Eta} is satisfied. Clearly,
			\begin{equation*}
				\Big|\frac{\rmd^k}{\rmd s^k}\eta (s) \Big| \leq C_k s^{-k} \eta(s),\qquad
				\Big|\frac{\rmd^k}{\rmd s^k}\omega (s) \Big| \leq C_k s^{-k} \omega(s)
			\end{equation*}
			for all $k \in \N$ and large $s \in \R$. The relations
			\begin{equation*}
				\omega(\jbl \xi + \zeta \jbr)\leq \omega(\jxi) + \omega(\jbl \zeta \jbr),\quad \text{ and }\quad \eta(\jbl \xi + \zeta \jbr)\leq \eta(\jxi) + \eta(\jbl \zeta \jbr)
			\end{equation*}
			for sufficiently large $\jxi$ and $\jbl \zeta \jbr$, can be proved by using a similar approach to the one used in Example~\ref{SH:MR:EXAMPLES:LogLogmLip}.
			
			Thus, all assumptions are satisfied and we  have a global (in time) solution $u = u(t,\,x) \in \bigcap\limits_{j = 0}^{m-1} C^{m-1-j}\big([0,T];\,H^{\nu+j}_{\eta,\,\delta}\big)$ with $\delta < \min\{\delta_0,\,\delta_1,\,\delta_2\}$. We expect an infinite loss of derivatives since $\log(s) = o(\eta(s))$.
		\end{Example}	
	
	\section{Proofs}\label{PROOF}
	
		In this section, we prove Theorem~\ref{SH:MR:THEOREM:STRONG:MainTheorem} and Theorem~\ref{SH:MR:THEOREM:WEAK:MainTheorem}.
		Both proofs follow the steps described below and are, in fact, identical until we perform the change of variables. In the case of strong moduli of continuity, this change of variables features pseudodifferential operators of finite order; for weak moduli of continuity these pseudodifferential operators are of infinite order.
	
		The first step of both proofs is to introduce regularized characteristic roots $\lambda_j$ which are smooth in time. We continue by rewriting the original differential equation using the newly defined regularized roots and transform the differential equation into a system of order one with respect to the derivatives in time. After the diagonalization of the principal part, we perform a change of variables (containing the loss of derivatives). Finally, we able to apply sharp G{\aa}rding's inequality and Gronwall's lemma to derive a $H^\nu$-$H^\nu$-estimate, this implies $L^2$ well-posedness for the auxiliary Cauchy problem.

	\subsection{Regularize Characteristic Roots}\label{SH:PROOF:ROOTS}
	
	The characteristic roots of the operator in \eqref{SH:MR:THEOREM:WEAK:CauchyProblem} are the solutions $\tau_1,\,\ldots,\,\tau_m$ of characteristic equation
	\begin{equation*}
		\tau^m - \sum\limits_{j=0}^{m-1}\sum\limits_{|\gamma| + j  = m} a_{m-j,\,\gamma}(t,\,x) \xi^\gamma \tau^j = 0.
	\end{equation*}

	They are homogeneous of degree $1$ in $\xi$ and
	\begin{equation*}
		\big|\partial_\xi^\alpha D_x^\beta \tau_j(t,\,x,\,\xi) - \partial_\xi^\alpha D_x^\beta \tau_j(s,\,x,\,\xi)\big| \leq C K_{|\beta|} \mu(|t-s|) \jxi^{1-|\alpha|},
	\end{equation*}
	due to the assumptions \ref{RESULTS:Both:1} and \ref{RESULTS:Both:2}.
	Since the characteristic roots are only $\mu$-continuous in time, it is useful to approximate them by regularized roots which are smooth in time.
	\begin{Definition}\label{SH:PROOF:ROOTS:DEF:Roots}
		Let $\vp \in \Czi(\R)$ be a given function satisfying $\int_\R \vp(x) \rmd x = 1$ and $\vp(x)~\geq~0$ for any $x \in \R$  with $\supp \vp \subset [{-1},1]$. Let $\ve > 0$ and set $\vp_\ve(x) = \frac{1}{\ve} \vp\left(\frac{x}{\ve}\right)$.
		Then we define for $j = 1,\,\ldots,\, m$,
		\begin{equation*}
			\lambda_j(t,\,x,\,\xi) := (\tau_j(\cdot,\,x,\,\xi) \ast \vp_\ve(\cdot))(t,\,x,\,\xi).
		\end{equation*}
	\end{Definition}
	
	\begin{Remark}\label{SH:PROOF:ROOTS:Remark:ExtendTau}
		The characteristic roots $\tau_j(t,\,x,\,\xi)$ are defined on $[0,\,T]\times\R^n\times\R^n$. For the above definition to be sensible, it is necessary to  extend them to $[-\ve,\,T+\ve]\times\R^n\times\R^n$, continuously. However, we do not distinguish between $\tau_j(t,\,x,\,\xi)$ and their continuation but just write $\tau_j(t,\,x,\,\xi)$.
	\end{Remark}
	The following relations are obtained by straightforward computations.
	\begin{Proposition}\label{SH:PROOF:ROOTS:Lemma:RegularRoots}
		For $\ve = \jxi^{-1}$ we have $\lambda_j \in C\big([0,\,T];\,\Sy^1\big)$. There exist constants $C_\alpha > 0$ such that the inequalities
		\begin{enumerate}[label = (\roman*),align = left, leftmargin=*]
			\item $\big|\partial_\xi^\alpha D_x^\beta \partial_t^k \lambda_j(t,\,x,\,\xi)\big| \leq C_\alpha K_{|\beta|} \jxi^{k+1-|\alpha|} \mu(\jxi^{-1})$, for all $k \geq 1$, $j = 1,\,\ldots,\,m$,
			\item $\big|\partial_\xi^\alpha D_x^\beta\big(\lambda_j(t,\,x,\,\xi) - \tau_j(t,\,x,\,\xi)\big)\big| \leq C_\alpha K_{|\beta|} \jbl\xi\jbr^{1-|\alpha|} \mu(\jxi^{-1}) $, for all $j = 1,\,\ldots,\,m$, and
			\item $\lambda_j(t,\,x,\,\xi) - \lambda_i(t,\,x,\,\xi) \geq C \jbl\xi\jbr$, for all $1 \leq i < j \leq m$,
		\end{enumerate}
		are satisfied for all $t \in [0,\,T]$ and $x,\,\xi \in \R^n$.
	\end{Proposition}
	\begin{Remark}\label{SH:PROOF:ROOTS:Remark:RegularRoots}
		We observe that $(i)$ and $(ii)$ in Proposition~\ref{SH:PROOF:ROOTS:Lemma:RegularRoots} are equivalent to
		\begin{enumerate}[label = (\roman*),align = left, leftmargin=*]
			\item $\big| \partial_\xi^\alpha D_x^\beta \partial_t^k \lambda_j(t,\,x,\,\xi)\big| \leq C_\alpha K_{|\beta|} \jxi^{k-|\alpha|} \omega(\jxi)$ and
			\item $\big|\partial_\xi^\alpha D_x^\beta\big(\lambda_j(t,\,x,\,\xi) - \tau_j(t,\,x,\,\xi)\big)\big| \leq C_\alpha K_{|\beta|} \jbl\xi\jbr^{-|\alpha|} \omega(\jxi)$.
		\end{enumerate}
	\end{Remark}
	
	\subsection{Factorization and Reduction to a Pseudodifferential System of First Order}\label{SH:PROOF:FAC}

	We are interested in a factorization of the operator $P(t,\,x,\,D_t,\,D_x)$. Formally, this leads to
	\begin{equation}\label{SH:PROOF:FAC:OpTau}
	\begin{aligned}
		P(t,\,x,\,D_t,\,D_x) = (D_t - \tau_m(t,\,x,\,D_x)) \cdots (D_t - \tau_1(t,\,x,\,D_x))\\+ \sum\limits_{j=0}^{m-1} R_j(t,\,x,\,D_x) D_t^j,
	\end{aligned}
	\end{equation}
	where the difficulty is that the operators $\tau_k(t,\,x,\,D_x)$ are not differentiable with respect to $t$, which means that the composition $D_t \circ \tau_k(t,\,x,\,D_x)$ may not be well-defined. An idea to overcome this difficulty is to use the regularized roots $\lambda_k(t,\,x,\,D_x)$ in \eqref{SH:PROOF:FAC:OpTau} instead of $\tau_k(t,\,x,\,D_x)$. This idea, however, comes at the price of increasing the order of the lower order terms. How much their order is increased depends on the terms $\big|\partial_\xi^\alpha D_x^\beta\big(\lambda_j(t,\,x,\,\xi) - \tau_j(t,\,x,\,\xi)\big)\big|$. In view of Proposition~\ref{SH:PROOF:ROOTS:Lemma:RegularRoots} and Remark~\ref{SH:PROOF:ROOTS:Remark:RegularRoots} we define
	the operator
	\begin{equation*}
		\widetilde P(t,\,x,\,D_t,\,D_x) = (D_t - \lambda_m(t,\,x,\,D_x) \circ \ldots \circ (D_t - \lambda_1(t,\,x,\,D_x)),
	\end{equation*}
	and observe that $\widetilde P(t,\,x,\,D_t,\,D_x)$ is a factorization of $P(t,\,x,\,D_t,\,D_x)$ in the sense that
	\begin{equation}\label{SH:PROOF:FAC:OpLambda}
	P(t,\,x,\,D_t,\,D_x) = \widetilde P(t,\,x,\,D_t,\,D_x) + \sum\limits_{j=0}^{m-1} R_j(t,\,x,\,D_x)D_t^j,
	\end{equation}
	where
	$R_j(t,\,x,\,D_x) \in C\big([0,\,T];\,\OPS^{m-1-j,\,\omega}\big)$, $j = 1,\,\ldots,\,m-1$.
	
	Our next step is to transform the differential equation $Pu = f$ into a pseudodifferential system of first order. For this purpose, we consider $P = P(t,\,x,\,D_t,\,D_x)$ as given in \eqref{SH:PROOF:FAC:OpLambda} and introduce the change of variables $U = U(t,\,x) = (u_0(t,\,x),\,\ldots,\,u_{m-1}(t,\,x))^T$, where
	\begin{align*}
		u_0(t,\,x) &= \jbl D_x \jbr^{m-1} v_0(t,\,x), && u_j(t,\,x) = \jbl D_x\jbr^{m-1-j}v_j,\\
		v_0(t,\,x) &= u(t,\,x),  &&v_j(t,\,x) = (D_t - \lambda_j(t,\,x,\,D_x))v_{j-1}(t,\,x),
	\end{align*}
	for $j = 1,\,\ldots,\,m-1$. By including the terms $D_t - \lambda_j(t,\,x,\,D_x)$ (and not just $D_t$) in the change of variables (i.e. also in the energy), the principal part of the operator of the resulting system is already almost diagonal. More precisely, this means that
	$Pu = f$
	is equivalent to
	\begin{equation}\label{SH:PROOF:FAC:System}
	D_t U(t,\,x) -
	A(t,\,x,\,D_x)
	U(t,\,x) + B(t,\,x,\,D_x)U(t,\,x) =
	(0, \,\ldots\, 0,\,	 f(t,\,x))^T,
	\end{equation}
	where
	\begin{equation*}
		A(t,\,x,\,D_x)  = \begin{pmatrix}
			\lambda_1(t,\,x,\,D_x) & \jbl D_x \jbr & 0	& \ldots & 0	\\
			0&\ddots	& \ddots & &\vdots\\
			\vdots&&\ddots	& \ddots &0\\
			\vdots& & &\ddots&\jbl D_x \jbr\\
			0 &\ldots&\ldots&0&\lambda_m(t,\,x,\,D_x)
		\end{pmatrix},
	\end{equation*}
	and $B(t,\,x,\,D_x) = \{b_{i,j}(t,\,x,\,D_x)\}_{1\leq i,\,j \leq m}$ is a matrix of lower order terms which satisfies
	\begin{equation*}
		b_{i,j}(t,\,x,\,D_x)= 0,
	\end{equation*}
	for $ i=1,\,\ldots,\,m-1$ and $j=1,\,\ldots,\,m$, and
	\begin{equation}\label{SH:PROOF:FAC:EstB}
	|\partial_\xi^\alpha D_x^\beta b_{m,j}(t,\,x,\,\xi)| \leq C_{\alpha,\,\beta} \omega(\jbl\xi\jbr),
	\end{equation}
	for $i = m$ and $j = 1,\,\ldots,\,m$.
	
	\subsection{Diagonalization Procedure}\label{SH:PROOF:DIA}
	
	The principal part $A(t,\,x,\,D_x)$ in \eqref{SH:PROOF:FAC:System} can be diagonalized and the regularized roots $\lambda_j$ can be replaced by the original characteristic roots $\tau_j$ since there are invertible operators $H$ and $H^{-1}$ which satisfy the following proposition.
	\begin{Proposition}\label{SH:PROOF:DIA:Lemma:Diagonalize}
		Consider the matrix $T(t,\,x,\,\xi) = \{\beta_{p,\,q}(t,\,x,\,\xi)\}_{0\leq p,\,q\leq m-1}$, where
		\begin{align*}
			&\beta_{p,\,q}(t,\,x,\,\xi) = 0, \qquad p \geq q;\\
			&\beta_{p,\,q}(t,\,x,\,\xi) = \frac{(1-\vp_1(\xi))\jbl\xi\jbr^{k-j}}{d_{p,\,q}(t,\,x,\,\xi)}, \qquad p < q;\\
			&d_{p,\,q}(t,\,x,\,\xi) = \prod\limits_{r=p}^{q-1} \big(\lambda_q(t,\,x,\,\xi) - \lambda_r(t,\,x,\,\xi)\big), \; \vp_1 \in \Czi,\; \vp_1 = 1 \text{ for } |\xi| \leq M,
		\end{align*}
		where $M$ is a large parameter.
		We define $H(t,\,x,\,D_x)$ and  $H^{-1}(t,\,x,\,D_x)$ to be the pseudodifferential operators with symbols
		\begin{align*}
			H(t,\,x,\,\xi) &= I + T(t,\,x,\,\xi), \text{ and }\\
			H^{-1}(t,\,x,\,\xi) &= I + \sum\limits_{j=1}^{m-1} (-1)^j T^j(t,\,x,\,\xi).
		\end{align*}
		Then the following assertions hold true.
		\begin{enumerate}[label = (\roman*)]
			
			\item The operators $H(t,\,x,\,D_x)$ and $H^{-1}(t,\,x,\,D_x)$ are in $C\big([0,\,T];\,\OPS^{0}\big)$.
			
			\item The composition $(H^{-1}\circ H)(t,\,x,\,D_x)$ satisfies
			\begin{equation*}
				H^{-1}(t,\,x,\,D_x)\circ H(t,\,x,\,D_x) = I + K(t,\,x,\,D_x),		
			\end{equation*}
			where $K(t,\,x,\,D_x) \in C\big([0,\,T];\,\OPS^{-1}\big)$.
			
			\item The operator $(D_t H)(t,\,x,\,D_x)$ belongs to $C\big([0,\,T];\,\OPS^{0,\,\omega}\big)$.
			
			\item The operator $\widehat A(t,\,x,\,D_x) = H^{-1}(t,\,x,\,D_x)\circ A(t,\,x,\,D_x) \circ H(t,\,x,\,D_x)$ belongs  to $C\big([0,\,T];\,\OPS^{1}\big)$ and its full symbol may be written as
			\begin{equation*}
				\widehat A(t,\,x,\,\xi) =
				\begin{pmatrix}
					\tau_1(t,\,x,\,\xi) &&\\
					& \ddots &\\
					& & \tau_m(t,\,x,\,\xi)
				\end{pmatrix} + M(t,\,x,\,\xi),
			\end{equation*}
			where $M(t,\,x,\,\xi) \in C\big([0,\,T];\,\Sy^{0,\,\omega}\big)$ is a lower order term.
			
			\item The operator $\widehat B(t,\,x,\,D_x) = H^{-1}(t,\,x,\,D_x)\circ B(t,\,x,\,D_x) \circ H(t,\,x,\,D_x)$ belongs to $C\big([0,\,T];\,\OPS^{0,\,\omega}\big)$.\label{SH:PROOF:DIA:Lemma:Diagonalize:LOT}
			
		\end{enumerate}
	\end{Proposition}

	We perform the described diagonalization by setting $\widehat{U} = \widehat{U}(t,\,x) = H^{-1}(t,\,x,\,D_x) U(t,\,x)$ and obtain that
	\begin{equation*}
		P(t,\,x,\,D_t,\,D_x) U(t,\,x)=\big(0, \,\ldots\, 0,\,	 f(t,\,x)\big)^T,
	\end{equation*}
	is equivalent to
	\begin{align*}
		&		
	\begin{aligned}[t]
			& D_t\circ H(t,\,x,\,D_x) \widehat{U}(t,\,x) 	- A(t,\,x,\,D_x)\circ H(t,\,x,\,D_x) \widehat{U}(t,\,x) \\
			& \qquad + B(t,\,x,\,D_x)\circ H(t,\,x,\,D_x) \widehat{U}(t,\,x)
		\end{aligned}\\
		& \qquad =
		\begin{aligned}[t]
			& H(t,\,x,\,D_x) D_t \widehat{U}(t,\,x) - A(t,\,x,\,D_x)\circ H(t,\,x,\,D_x) \widehat{U}(t,\,x) \\
			& + (D_t H)(t,\,x,\,D_x) \widehat{U}(t,\,x)+ B(t,\,x,\,D_x)\circ H(t,\,x,\,D_x) \widehat{U}(t,\,x)
		\end{aligned}\\
		&\qquad =\big(0, \,\ldots\, 0,\,	 f(t,\,x)\big)^T.
	\end{align*}
	We set $F = F(t,\,x) = \big(0, \,\ldots\, 0,\,	 f(t,\,x)\big)^T$ and apply $H^{-1}(t,\,x,\,D_x)$ to both sides of the previous pseudodifferential system to obtain
	\begin{align}
		&		
		\begin{aligned}[m]
			&D_t \widehat{U}(t,\,x) + K(t,\,x,\,D_x) D_t \widehat{U}(t,\,x)\\
			&\qquad - H^{-1}(t,\,x,\,D_x)\circ A(t,\,x,\,D_x)\circ H(t,\,x,\,D_x) \widehat{U}(t,\,x) \\
			&\qquad +  H^{-1}(t,\,x,\,D_x)\circ (D_t H)(t,\,x,\,D_x) \widehat{U}(t,\,x)\\
			&\qquad +  H^{-1}(t,\,x,\,D_x)\circ B(t,\,x,\,D_x)\circ H(t,\,x,\,D_x) \widehat{U}(t,\,x) \\
		&\quad =H^{-1}(t,\,x,\,D_x) F(t,\,x).\label{SH:PROOF:DIA:Eq:1}
		\end{aligned}
	\end{align}
	Applying Proposition~\ref{SH:PROOF:DIA:Lemma:Diagonalize} yields
	\begin{align*}
		H^{-1}(t,\,x,\,D_x)\circ A(t,\,x,\,D_x)\circ H(t,\,x,\,D_x) &= \xbar A(t,\,x,\,D_x) + M(t,\,x,\,D_x),\\
		H^{-1}(t,\,x,\,D_x)\circ B(t,\,x,\,D_x)\circ H(t,\,x,\,D_x) &= \widehat B(t,\,x,\,D_x) \in  C\big([0,\,T];\,\OPS^{0,\,\omega}\big), \\
		H^{-1}(t,\,x,\,D_x)\circ (D_t H)(t,\,x,\,D_x) &\in C\Big([0,\,T];\,\OPS^{0,\,\omega}\big),
	\end{align*}
	where $ M(t,\,x,\,D_x)\in  C\big([0,\,T];\,\OPS^{0,\,\omega}\big)$ and
	\begin{equation*}
		\xbar A(t,\,x,\,D_x) =
		\begin{pmatrix}
			\tau_1(t,\,x,\,D_x) &&\\
			& \ddots &\\
			& & \tau_m(t,\,x,\,D_x)
		\end{pmatrix}
	\end{equation*}
	is diagonal. Hence, setting \[ \xbar B(t,\,x,\,D_x) = \widehat{B}(t,\,x,\,D_x) - M(t,\,x,\,D_x) + H^{-1}(t,\,x,\,D_x)\circ (D_t H)(t,\,x,\,D_x),\] the pseudodifferential system  \eqref{SH:PROOF:DIA:Eq:1} may be written as
	\begin{equation*}
		\begin{aligned}[t]
			&\big(I + K(t,\,x,\,D_x)\big) D_t \widehat{U}(t,\,x)
			- \xbar A(t,\,x,\,D_x)\widehat{U}(t,\,x) +\xbar B(t,\,x,\,D_x) \widehat{U}(t,\,x)\\
			& \quad =H^{-1}(t,\,x,\,D_x) F(t,\,x).
		\end{aligned}
	\end{equation*}
	We now observe, that $I + K(t,\,x,\,D_x)$ is invertible for sufficiently large values of $M$. Recall, that $K(t,\,x,\,D_x) \in C\big([0,\,T];\,\OPS^{-1}\big)$. Therefore we may estimate the operator norm of $K(t,\,x,\,D_x)$ by $C M^{-1}$, where $M$ is the parameter introduced in the definition of $H(t,\,x,\,D_x)$ in Proposition~\ref{SH:PROOF:DIA:Lemma:Diagonalize}. Choosing $M$ sufficiently large ensures that the operator norm of $K(t,\,x,\,D_x)$ is strictly smaller than $1$, which guarantees the existence of
	\begin{equation*}
		\big(I + K(t,\,x,\,D_x)\big)^{-1} = \sum\limits_{k=0}^{\infty} (-K(t,\,x,\,D_x))^k \in C\big([0,\,T];\,\OPS^{0}\big).
	\end{equation*}
	This means, that
	\begin{equation*}
		P(t,\,x,\,D_t,\,D_x) U(t,\,x)=\big(0, \,\ldots\, 0,\,	 f(t,\,x)\big)^T
	\end{equation*}
	is equivalent to
	\begin{equation}\label{SH:PROOF:DIA:Eq:2}
	L(t,\,x,\,D_t,\,D_x) \widehat{U}(t,\,x) = \widetilde H^{-1}(t,\,x,\,D_x) F(t,\,x),
	\end{equation}
	where
	\begin{equation*}
		L(t,\,x,\,D_t,\,D_x) =D_t - \widetilde{A}(t,\,x,\,D_x) + \widetilde{B}(t,\,x,\,D_x),
	\end{equation*}
	with
	\begin{align*}
		\widetilde H^{-1}(t,\,x,\,D_x) &= \big(I + K(t,\,x,\,D_x)\big)^{-1} \circ H^{-1}(t,\,x,\,D_x) \in C\big([0,\,T];\, \OPS^0\big),\\
		\widetilde A(t,\,x,\,D_x) &= \big(I + K(t,\,x,\,D_x)\big)^{-1}\circ  \xbar A(t,\,x,\,D_x)\in C\big([0,\,T];\, \OPS^1\big),\\
		\widetilde B(t,\,x,\,D_x) &= \big(I + K(t,\,x,\,D_x)\big)^{-1}\circ  \xbar B(t,\,x,\,D_x)\in C\big([0,\,T];\, \OPS^{0,\,\omega}\big).
	\end{align*}
	
	From this point on, the proofs of Theorem~\ref{SH:MR:THEOREM:STRONG:MainTheorem} and Theorem~\ref{SH:MR:THEOREM:WEAK:MainTheorem} differ. We first discuss how to proceed in the case of strong moduli of continuity and then conclude this chapter by finishing the proof for weak moduli of continuity as well.
	
		\subsection{Conjugation for Strong Moduli of Continuity}\label{SH:PROOF:CON:STRONG}
		
		Before we apply sharp G{\aa}rding's inequality and Gronwall's lemma we perform another change of variables, which allows us to control the lower order terms $\widetilde B(t,\,x,\,D_x)$.
		
		We set $\widehat U(t,\,x) = \jbl D_x \jbr^{-\nu} e^{\kappa t \omega(\jbl D_x\jbr)} V(t,\,x)$, $t\in[0,\,T]$, where $\kappa$ is a suitable positive constant, which is determined later and $\nu$ is the index of the Sobolev space $H^\nu$ which is related to the space in which we want to have well-posedness. We obtain that the pseudodifferential system  \eqref{SH:PROOF:DIA:Eq:2} is equivalent to
		\begin{equation}\label{SH:PROOF:CON:STRONG:Eq:1}
		\widetilde L(t,\,x,\,D_t,\,D_x) V(t,\,x) =\jbl D_x \jbr^{\nu} e^{-\kappa t \omega(\jbl D_x\jbr)} \widetilde H^{-1}(t,\,x,\,D_x) F(t,\,x),
		\end{equation}
		where
		\begin{align*}
			& \widetilde L(t,\,x,\,D_t,\,D_x) = D_t - \jbl D_x \jbr^{\nu} e^{-\kappa t \omega(\jbl D_x\jbr)} \circ\widetilde A(t,\,x,\,D_x) \circ e^{\kappa t \omega(\jbl D_x\jbr)} \jbl D_x \jbr^{-\nu}\\
			& \qquad  + \jbl D_x \jbr^{\nu} e^{-\kappa t \omega(\jbl D_x\jbr)}\circ \widetilde B(t,\,x,\,D_x)\circ e^{\kappa t \omega(\jbl D_x\jbr)}\jbl D_x \jbr^{-\nu} - \I \kappa \omega(\jbl D_x \jbr).
		\end{align*}
	
		Applying the composition rule for pseudodifferential operators of finite order, we obtain that
		\begin{align*}
			\jbl D_x \jbr^{\nu} e^{-\kappa t \omega(\jbl D_x\jbr)} \circ\widetilde A \circ e^{\kappa t \omega(\jbl D_x\jbr)} \jbl D_x \jbr^{-\nu} &= \widetilde A(t,\,x,\,D_x) + B_1(t,\,x,\,D_x),\\
			\jbl D_x \jbr^{\nu} e^{-\kappa t \omega(\jbl D_x\jbr)}\circ \widetilde B\circ e^{\kappa t \omega(\jbl D_x\jbr)}\jbl D_x \jbr^{-\nu}  &= \widetilde B(t,\,x,\,D_x) + B_2(t,\,x,\,D_x),
		\end{align*}
		where $B_1,\,B_2 \in C\big([0,\,T]; \OPS^0\big)$.
		
		We conclude that, for $t\in [0,\,T]$, the pseudodifferential system \eqref{SH:PROOF:DIA:Eq:2} is equivalent to
		\begin{equation}\label{SH:PROOF:CON:STRONG:Eq:3}
		\widetilde L(t,\,x,\,D_t,\,D_x) V(t,\,x) = \jbl D_x \jbr^{\nu} e^{-\kappa t \omega(\jbl D_x\jbr)} \widetilde H^{-1}(t,\,x,\,D_x) F(t,\,x),
		\end{equation}
		where
		\begin{align*}
			\widetilde L(t,\,x,\,D_t,\,D_x) &= D_t - \widetilde A(t,\,x,\,D_x) + \check B(t,\,x,\,D_x) - \I \kappa \omega(\jbl D_x \jbr),
		\end{align*}
		and $\check B \in C\big([0,\,T]; \OPS^{0,\,\omega}\big)$.
		
		\subsection{Well-Posedness of an Auxiliary Cauchy Problem for Strong Moduli of Continuity}\label{SH:PROOF:AUXCP:STRONG}
		
		We consider the auxiliary Cauchy problem
		\begin{equation}\label{SH:PROOF:AUXCP:STRONG:CP}
			\begin{aligned}
				\partial_t V =	 \big(\I \widetilde A(t,\,x,\,D_x) - \I \check B(t,\,x,\,D_x) - \kappa \omega(\jbl D_x \jbr)\big) V \\+ \I \jbl D_x \jbr^{\nu}e^{\rho(t) \omega(\jbl D_x \jbr)} \widetilde H^{-1}(t,\,x,\,D_x) F(t,\,x),
			\end{aligned}
		\end{equation}
		for $(t,\,x) \in [0,\,T]\times \R^n$,
		with initial conditions
		\begin{align*}
			V(0,\,x) =
			\big(
			v_0(x),\, \ldots,\,	v_{m-1}(x)\big)^T,
		\end{align*}
		where
		\begin{align*}
			v_j(x) &= \jbl D_x \jbr^\nu e^{-\kappa t \omega(\jbl D_x \jbr)} H^{-1}(0,\,x,\,D_x) \jbl D_x \jbr^{m-1-j}\\ &\qquad\times (D_t-  \lambda_j(0,\,x,\,D_x))\cdots(D_t - \lambda_1(0,\,x,\,D_x)) u(0,\,x)
		\end{align*}
		for $j = 0,\,\ldots m-1$.
		
		Recalling that $\partial_t \|V\|^2_{L^2}  = 2 \Re\big[(\partial_t V,\,V)_{L^2}\big]$ we obtain
		\begin{equation*}
			\begin{aligned}
			\partial_t\|V\|_{L^2}^2 = &2 \Re\big[((\I \widetilde A - \I\check B - \kappa\omega(\jbl D_x \jbr))V,\,V)_{L^2}\big]
			\\+&2 \Re\big[(\jbl D_x \jbr^{\nu}e^{-\kappa t \omega(\jbl D_x \jbr)} \widetilde H^{-1} F,\,V)_{L^2}\big].
			\end{aligned}
		\end{equation*}
		We observe that
		\begin{align*}
			\begin{aligned}
			 \Re\big[\I \widetilde A(t,\,x,\,\xi) - \I\check B(t,\,x,\,\xi) - \kappa\omega(\jbl \xi \jbr)\big] = 	-\Re\big[\I\check B(t,\,x,\,\xi) + \kappa\omega(\jbl \xi \jbr)\big],
			\end{aligned}
		\end{align*}
		since  $\Re\big[\I \widetilde A(t,\,x,\,\xi)\big] = 0$ (by assumption~\ref{SH:MR:THEOREM:STRONG:ASSUME:StrictHyp}).
		Taking account of $\check B \in C\big([0,\,T]; \OPS^{0,\,\omega}\big)$ it follows
		\begin{equation*}
			\Re\big[\I\check B(t,\,x,\,\xi) + \kappa\omega(\jbl \xi \jbr)\big] \geq \Re\big[-C_B \omega(\jxi) + \kappa\omega(\jbl \xi \jbr)\big] \geq 0
		\end{equation*}
		if we choose $\kappa > C_B$, where $C_B >0$ depends on $\check B$ and, hence, it is determined by the coefficients of the original differential equation. Thus, we are able to apply sharp G{\aa}rding's inequality to obtain
		\begin{align*}
			2 \Re\big[\big((\I \widetilde A - \I\check B - \kappa\omega(\jbl D_x \jbr))V,\,V\big)_{L^2}\big]	\leq C \| V \|^2_{L^2}.
		\end{align*}
		This yields
		\begin{equation} \label{SH:PROOF:AUXCP:STRONG:BeforeGronwall}
		\partial_t\|V\|_{L^2}^2 \leq C \|V\|_{L^2}^2 + C \| \jbl D_x \jbr^{\nu} e^{-\kappa t \omega(\jbl D_x \jbr)} \widetilde H^{-1} F\|_{L^2}^2.
		\end{equation}
	%
		We apply Gronwall's Lemma to \eqref{SH:PROOF:AUXCP:STRONG:BeforeGronwall} and obtain
		\begin{equation}\label{SH:PROOF:AUXCP:STRONG:Gronwall}
			\begin{aligned}[m]
				\|V(t,\,\cdot)\|_{L^2}^2 &\leq C \|V(0,\,\cdot)\|_{L^2}^2\\&+ C \int\limits_0^t \| \jbl D_x \jbr^{\nu} e^{-\kappa z \omega(\jbl D_x\jbr)} \widetilde H^{-1}(z,\,x,\,D_x) F(z,\,\cdot)\|_{L^2}^2 \rmd z,
			\end{aligned}
		\end{equation}
		for $t\in[0,\,T]$.
		We recall that $\widetilde H^{-1}(t,\,x,\,D_x)$ is a pseudo-differential operator of order zero and use assumption~\ref{SH:MR:THEOREM:STRONG:ASSUME:Inhomogeneity} to obtain that
		\begin{equation*}
			\| \jbl D_x \jbr^\nu e^{-\kappa z \omega(\jbl D_x\jbr)} \widetilde  H^{-1}(z,\,x,\,D_x) F(z,\,\cdot)\|_{L^2}^2  \leq C_{s,\,z} < \infty,
		\end{equation*}
		similarly, in view of assumption~\ref{SH:MR:THEOREM:STRONG:ASSUME:DATA}, it is clear that,
		\begin{equation*}
			\|V(0,\,\cdot)\|_{L^2}^2 =  \|\jbl D_x \jbr^{\nu} \widetilde  H^{-1}(0,\,x,\,D_x) U(0,\,x)\|_{L^2}^2 \leq C < \infty.
		\end{equation*}
		
		We conclude that
		\begin{equation*}
			\begin{aligned}
				\|\jbl D_x \jbr^{\nu} e^{-\kappa t \omega(\jbl D_x\jbr)} &\widehat U(t,\,\cdot)\|_{L^2}^2 \leq C \| \jbl D_x \jbr^\nu \widehat U(0,\,\cdot)\|_{L^2}^2\\&+ C \int\limits_0^t \| \jbl D_x \jbr^{\nu} e^{-\kappa z \omega(\jbl D_x\jbr)} \widetilde H^{-1}(z,\,x,\,D_x) F(z,\,\cdot)\|_{L^2}^2 \rmd z,
			\end{aligned}
		\end{equation*}
		which means that the solution $\widehat U$ to \eqref{SH:PROOF:DIA:Eq:2} belongs to $C([0,\,T]; H^\nu_{\omega,\,-\kappa T})$.
		Returning to our original solution $u = u(t,\,x)$ we obtain that
		\begin{equation*}
			\begin{aligned}[t]
				\sum\limits_{j = 0}^{m-1} \big\|\jbl D_x \jbr^\nu e^{- \kappa t \omega(\jbl D_x \jbr)} \partial_t^j u(t,\,\cdot)\big\|^2_{L^2} \leq &C \Big(\sum\limits_{j = 0}^{m-1} \big\|\jbl D_x \jbr^\nu g_j(0,\,\cdot)\big\|^2_{L^2}\\& + \int\limits_0^t \big\|\jbl D_x \jbr^\nu e^{- \kappa z \omega(\jbl D_x \jbr)} f(z,\,\cdot)\big\|^2_{L^2} \rmd z\Big),
			\end{aligned}
		\end{equation*}
		for $0 \leq t \leq T$ and some $C = C_\nu > 0$.
		This, in turn means that the original Cauchy problem \eqref{SH:MR:THEOREM:WEAK:MainTheorem:Equ} is well-posed for $u=u(t,\,x)$, with
		\begin{equation*}
			u \in \bigcap\limits_{j = 0}^{m-1} C^{m-1-j}([0,\,T];\, H^{\nu+j}_{\omega,\,-\kappa T}).
		\end{equation*}
		This concludes the proof for strong moduli of continuity.

	\subsection{Conjugation for Weak Moduli of Continuity}\label{SH:PROOF:CON:WEAK}
	
	In Section~\ref{SH:PROOF:DIA} we left off by observing that
	\begin{equation*}
		P(t,\,x,\,D_t,\,D_x) U(t,\,x)=\big(0, \,\ldots\, 0,\,	 f(t,\,x)\big)^T
	\end{equation*}
	is equivalent to
	\begin{equation} \tag{\ref{SH:PROOF:DIA:Eq:2}}
		L(t,\,x,\,D_t,\,D_x) \widehat{U}(t,\,x) = \widetilde H^{-1}(t,\,x,\,D_x) F(t,\,x),
	\end{equation}
	where
	\begin{equation*}
	L(t,\,x,\,D_t,\,D_x) =D_t - \widetilde{A}(t,\,x,\,D_x) + \widetilde{B}(t,\,x,\,D_x)
	\end{equation*}
	with
	\begin{align*}
		\widetilde H^{-1}(t,\,x,\,D_x) &\in C\big([0,\,T];\, \OPS^0\big),\\
		\widetilde B(t,\,x,\,D_x) &\in C\big([0,\,T];\, \OPS^{0,\,\omega}\big), \text{ and}\\
		\widetilde A(t,\,x,\,D_x) &= \begin{pmatrix}
			\tau_1(t,\,x,\,D_x) &&\\
			& \ddots &\\
			& & \tau_m(t,\,x,\,D_x)
		\end{pmatrix}
		\in C\big([0,\,T];\, \OPS^1\big)
	\end{align*}
is diagonal. As in the case of strong moduli of continuity, this time we also perform a change of variables to control the lower order terms $\widetilde B(t,\,x,\,D_x)$. However, this time the involved pseudodifferential operators are of infinite order.	
	
	We set \[ \widehat U(t,\,x) = \jbl D_x \jbr^{-\nu} e^{-\kappa(T^\ast-t) \omega(\jbl D_x\jbr)} V(t,\,x), \,t\in[0,\,T^\ast], \]
where $\kappa$ and $T^\ast$ are suitable positive constants, which are determined later and $s \in \R$. We obtain that the system \eqref{SH:PROOF:DIA:Eq:2} is equivalent to
	\begin{equation}\label{SH:PROOF:CON:WEAK:Eq:1}
	\widetilde L(t,\,x,\,D_t,\,D_x) V(t,\,x) =\jbl D_x \jbr^{\nu} e^{\kappa(T^\ast-t) \omega(\jbl D_x\jbr)} \widetilde H^{-1}(t,\,x,\,D_x) F(t,\,x),
	\end{equation}
	where
	\begin{align*}
		\widetilde L(t,\,x,\,D_t,\,D_x) = D_t &- \jbl D_x \jbr^{\nu}e^{\kappa(T^\ast-t) \omega(\jbl D_x\jbr)} \circ\widetilde A \circ e^{-\kappa(T^\ast-t) \omega(\jbl D_x\jbr)}\jbl D_x \jbr^{-\nu}\\
		&  + \jbl D_x \jbr^{\nu}e^{\kappa(T^\ast-t) \omega(\jbl D_x\jbr)}\circ \widetilde B\circ e^{-\kappa(T^\ast-t) \omega(\jbl D_x\jbr)}\jbl D_x \jbr^{-\nu}\\ &- \I \kappa \omega(\jbl D_x \jbr).
	\end{align*}
	\begin{Remark}
		We note that the operator $e^{\kappa(T^\ast-t) \omega(\jbl D_x\jbr)} $ is of infinite order since $o(\omega(s)) = \log(s)$. Therefore, we cannot apply the standard asymptotic expansion of the product of pseudodifferential operators of finite order.
	\end{Remark}
	
	We write $\rho(t) = \kappa(T^\ast-t)$ and choose $T^\ast$ such that the conjugation condition
	\begin{equation*}
		\rho(t) = \kappa(T^\ast-t) \leq \kappa T^\ast  < \delta_0
	\end{equation*}
	is satisfied for all $t \in [0,\,T^\ast]$, where $\delta_0$ is the constant given by assumption \ref{SH:MR:THEOREM:WEAK:ASSUME:ConEtaK} and $\kappa > 0$ is determined later on.
	In view of condition~\eqref{SH:MR:THEOREM:WEAK:MainTheorem:Equ} and by assumption~\ref{SH:MR:THEOREM:WEAK:ASSUME:ConEtaK} it is clear that
	\begin{equation}\label{SH:PROOF:CON:WEAK:OmegaK}
	\inf\limits_{l \in \N} \frac{K_l}{\jxi^l} \leq C e^{-\delta_0 \eta(\jxi)} \leq C e^{-\delta_0 \omega(\jxi)}.
	\end{equation}
	Assumption~\ref{SH:MR:THEOREM:WEAK:ASSUME:Eta} and \eqref{SH:PROOF:CON:WEAK:OmegaK} enable us to apply Proposition~\ref{APP:PSEUDO:CALC:CON} and Proposition~\ref{APP:PSEUDO:CALC:CON:EstR}, which allow us to conclude that
	\begin{equation}\label{SH:PROOF:CON:WEAK:Eq2}
	\begin{aligned}[m]
	\widetilde A_\omega(t,\,x,\,D_x) &= \jbl D_x \jbr^{\nu}e^{\rho(t) \omega(\jbl D_x\jbr)} \circ \widetilde A(t,\,x,\,D_x) \circ e^{-\rho(t) \omega(\jbl D_x\jbr)}\jbl D_x \jbr^{-\nu}, \text{ and }\\
	\widetilde B_\omega(t,\,x,\,D_x) &= \jbl D_x \jbr^{\nu}e^{\rho(t) \omega(\jbl D_x\jbr)} \circ \widetilde B(t,\,x,\,D_x) \circ e^{-\rho(t) \omega(\jbl D_x\jbr)}\jbl D_x \jbr^{-\nu},
	\end{aligned}
	\end{equation}
	satisfy
	\begin{align*}
		\widetilde A_\omega(t,\,x,\,\xi) &= \widetilde A(t,\,x,\,\xi) + \sum\limits_{0 < |\gamma| < N} D_x^\gamma \widetilde A(t,\,x,\,\xi) \chi_\gamma(\xi) + r_N(\widetilde A;\,x,\,\xi), \text{ and }\\
		\widetilde B_\omega(t,\,x,\,\xi) &= \widetilde B(t,\,x,\,\xi) + \sum\limits_{0 < |\gamma| < N} D_x^\gamma \widetilde B(t,\,x,\,\xi) \chi_\gamma(\xi) + r_N(\widetilde B;\,x,\,\xi),
	\end{align*}
	where
	\begin{equation*}
		|\partial_\xi^\alpha \chi_\gamma(\xi)| \leq C_{\alpha,\,\gamma} \rho(t)^{|\gamma|} \jxi^{-|\alpha|-|\gamma|}(\omega(\jxi))^{|\gamma|},
	\end{equation*}
	and
	\begin{align*}
		\big|\partial_\xi^\alpha D_x^\beta r_N(\widetilde A;\,x,\,\xi)\big| &\leq C_{\alpha,\,\beta,\,N}  \rho(t)^{N}\jxi^{1-|\alpha|} \Big(\frac{\omega(\jxi)}{\jxi}\Big)^N,\text{ and }\\
		\big|\partial_\xi^\alpha D_x^\beta r_N(\widetilde B;\,x,\,\xi)\big| &\leq C_{\alpha,\,\beta,\,N}  \rho(t)^{N}\jxi^{-|\alpha|} \omega(\jxi) \Big(\frac{\omega(\jxi)}{\jxi}\Big)^N,
	\end{align*}
	for $t \in [0,\,T^\ast]$.
	Thus,
	\begin{align*}
		\widetilde A_\omega(t,\,x,\,\xi) &= \widetilde A(t,\,x,\,\xi) + \sum\limits_{0 < |\gamma| \leq N} R_{\gamma}(\widetilde A;\,t,\,x,\,\xi), \text{ and }\\
		\widetilde B_\omega(t,\,x,\,\xi) &= \widetilde B(t,\,x,\,\xi) + \sum\limits_{0 < |\gamma| \leq N} R_{\gamma}(\widetilde B;\,t,\,x,\,\xi),
	\end{align*}
	where
	\begin{align*}
		R_{\gamma}(\widetilde A;\,t,\,x,\,\xi) &\in C\big([0,\,T];\Sy^{1-|\gamma|,\,\omega^{|\gamma|}}\big) \subset C\big([0,\,T^\ast];\Sy^{0,\,\omega}\big),\text{ and }\\
		R_{\gamma}(\widetilde B;\,t,\,x,\,\xi) &\in C\big([0,\,T];\Sy^{-|\gamma|,\,\omega^{|\gamma|+1}}\big) \subset C\big([0,\,T^\ast];\Sy^{0,\,\omega}\big).
	\end{align*}
	The above observations show that the conjugation we perform in \eqref{SH:PROOF:CON:WEAK:Eq2}, does not change the principal part of the operators $\widetilde A(t,\,x,\,D_x)$ and $\widetilde B(t,\,x,\,D_x)$ but introduces more lower order terms. We denote these new lower order terms by
	\begin{equation*}
		R(\widetilde A, \widetilde B;\,t,\,x,\,D_x) =  \sum\limits_{0 < |\gamma| \leq N} R_{\gamma}(\widetilde A;\,t,\,x,\,D_x)+ R_{\gamma}(\widetilde B;\,t,\,x,\,D_x),
	\end{equation*}
	and set
	\begin{equation*}
		\check B(t,\,x,\,D_x) = \widetilde B(t,\,x,\,D_x) +  R(\widetilde A, \widetilde B;\,t,\,x,\,D_x) \in C\big([0,\,T^\ast];\OPS^{0,\,\omega}\big)
	\end{equation*}
	with
	\begin{equation}\label{SH:PROOF:CON:WEAK:EstB}
	|\partial_\xi^\alpha D_x^\beta \check B(t,\,x,\,\xi)| \leq C_{\alpha,\,\beta} (1+\rho(t))\jxi^{-|\alpha|} \omega(\jxi).
	\end{equation}
	We conclude that, for $t\in [0,\,T^\ast]$, the system \eqref{SH:PROOF:DIA:Eq:2} is equivalent to
	\begin{equation}\label{SH:PROOF:CON:WEAK:Eq:3}
	\widetilde L(t,\,x,\,D_t,\,D_x) V(t,\,x) =\jbl D_x \jbr^{\nu}e^{\kappa(T^\ast-t) \omega(\jbl D_x\jbr)} \widetilde H^{-1}(t,\,x,\,D_x) F(t,\,x),
	\end{equation}
	where
	\begin{align*}
		\widetilde L(t,\,x,\,D_t,\,D_x) &= D_t - \widetilde A(t,\,x,\,D_x) + \check B(t,\,x,\,D_x) - \I \kappa \omega(\jbl D_x \jbr).
	\end{align*}
	Rearranging the system \eqref{SH:PROOF:CON:WEAK:Eq:3} yields
	\begin{equation*}
		\begin{aligned}
			\partial_t V =	 \big(\I \widetilde A(t,\,x,\,D_x) - \I \check B(t,\,x,\,D_x) - \kappa \omega(\jbl D_x \jbr)\big) V \\+ \I \jbl D_x \jbr^{\nu}e^{\rho(t) \omega(\jbl D_x \jbr)} \widetilde H^{-1}(t,\,x,\,D_x) F(t,\,x).
		\end{aligned}
	\end{equation*}
	
	\subsection{Well-Posedness of an Auxiliary Cauchy Problem for Weak Moduli of Continuity}\label{SH:PROOF:AUXCP:WEAK}
	
	In this section we consider the auxiliary Cauchy problem
	\begin{equation}\label{SH:PROOF:AUXCP:WEAK:CP}
		\begin{aligned}
			\partial_t V =	 \big(\I \widetilde A(t,\,x,\,D_x) - \I \check B(t,\,x,\,D_x) - \kappa \omega(\jbl D_x \jbr)\big) V \\+ \I \jbl D_x \jbr^{\nu}e^{\rho(t) \omega(\jbl D_x \jbr)} \widetilde H^{-1}(t,\,x,\,D_x) F(t,\,x).
		\end{aligned}
	\end{equation}
	for $(t,\,x) \in [0,\,T^\ast]\times \R^n$,
	with initial condition
	\begin{align*}
		V(0,\,x) =
		\big(
		v_0(x),\, \ldots,\,	v_{m-1}(x)\big)^T,
	\end{align*}
	where
	\begin{align*}
		v_j(x) &= \jbl D_x \jbr^{\nu} e^{\rho(0)\omega(\jbl D_x \jbr)} H^{-1}(0,\,x,\,D_x) \jbl D_x \jbr^{m-1-j}\\ &\qquad\times (D_t-  \lambda_j(0,\,x,\,D_x))\cdots(D_t - \lambda_1(0,\,x,\,D_x)) u(0,\,x),
	\end{align*}
	for $j = 0,\,\ldots m-1$.
	
	Recalling that $\partial_t \|V\|^2_{L^2}  = 2 \Re\big[(\partial_t V,\,V)_{L^2}\big]$ we obtain
	\begin{equation*}
		\begin{aligned}
			\partial_t\|V\|_{L^2}^2 = &2 \Re\big[\big((\I \widetilde A - \I\check B - \kappa\omega(\jbl D_x \jbr))V,\,V\big)_{L^2}\big]
			\\+&2 \Re\big[\big(\jbl D_x \jbr^{\nu}e^{\rho(t) \omega(\jbl D_x \jbr)} \widetilde H^{-1} F,\,V\big)_{L^2}\big].
		\end{aligned}
	\end{equation*}
	We observe that
	\begin{equation*}
		\Re\big[\I \widetilde A(t,\,x,\,\xi) - \I\check B(t,\,x,\,\xi) - \kappa\omega(\jbl \xi \jbr)\big]=	-\Re\big[\I\check B(t,\,x,\,\xi) + \kappa\omega(\jbl \xi \jbr)\big],
	\end{equation*}
	since  $\Re[\I \widetilde A(t,\,x,\,\xi)] = 0$ (by assumption~\ref{RESULTS:Both:1}).
	In view of \eqref{SH:PROOF:CON:WEAK:EstB} we obtain
	\begin{align*}
		\Re\big[\I\check B(t,\,x,\,\xi) + \kappa\omega(\jbl \xi \jbr)\big] &\geq \Re\big[-C_B (1+\rho(t)) \omega(\jxi) + \kappa\omega(\jbl \xi \jbr)\big]\\
		&\geq \Re\big[-C_B (1+ \kappa T^\ast) \omega(\jxi) + \kappa\omega(\jbl \xi \jbr)\big] \geq 0,
	\end{align*}
	if we choose $T^\ast$ such that $T^\ast C_B < 1$ and $\kappa$ is sufficiently large. Thus, we are able to apply sharp G{\aa}rding's inequality to obtain
	\begin{align*}
		2 \Re\big[\big((\I \widetilde A - \I\check B - \kappa\omega(\jbl D_x \jbr))V,\,V\big)_{L^2}\big]	\leq C \| V \|^2_{L^2}.
	\end{align*}
	This yields
	\begin{equation} \label{SH:PROOF:AUXCP:WEAK:BeforeGronwall}
	\partial_t\|V\|_{L^2}^2 \leq C \|V\|_{L^2}^2 + C \big\| \jbl D_x \jbr^{\nu} e^{\rho(t) \omega(\jbl D_x \jbr)} \widetilde H^{-1} F\big\|_{L^2}^2.
	\end{equation}
	
	We apply Gronwall's Lemma to \eqref{SH:PROOF:AUXCP:WEAK:BeforeGronwall} and obtain
	\begin{equation}\label{SH:PROOF:AUXCP:WEAK:Gronwall}
	\begin{aligned}
		\|V(t,\,\cdot)\|_{L^2}^2 \leq &C \|V(0,\,\cdot)\|_{L^2}^2 \\&+ C \int\limits_0^t \| \jbl D_x \jbr^{\nu} e^{\rho(z) \omega(\jbl D_x\jbr)} \widetilde H^{-1}(z,\,x,\,D_x) F(z,\,\cdot)\|_{L^2}^2 \rmd z,
	\end{aligned}
	\end{equation}
	for $t\in[0,\,T^\ast]$.
	We recall that $\widetilde H^{-1}(t,\,x,\,D_x)$ is a pseudodifferential operator of order zero and use assumption~\ref{SH:MR:THEOREM:WEAK:ASSUME:Inhomogeneity} to obtain that
	\begin{equation*}
		\| \jbl D_x \jbr^\nu e^{\kappa(T^\ast - z) \omega(\jbl D_x\jbr)} \widetilde  H^{-1}(z,\,x,\,D_x) F(z,\,\cdot)\|_{L^2}^2  \leq C_z < \infty,
	\end{equation*}
	similarly, in view of assumption~\ref{SH:MR:THEOREM:WEAK:ASSUME:DATA}, it is clear that,
	\begin{equation*}
		\|V(0,\,\cdot)\|_{L^2}^2 =  \|\jbl D_x \jbr^{\nu} e^{\kappa T^\ast \omega(\jbl D_x\jbr)} \widetilde  H^{-1}(0,\,x,\,D_x) U(0,\,x)\|_{L^2}^2 \leq C < \infty.
	\end{equation*}
	
	At the moment, we only know that
	\begin{equation*}
		\|V(t,\,\cdot)\|_{L^2}^2 = \|\jbl D_x \jbr^\nu e^{\rho(t) \omega(\jbl D_x \jbr)} \widetilde H^{-1}(t,\,x,\,D_x) U(t,\,\cdot)\|_{L^2}^2 \leq C < \infty,
	\end{equation*}
	for $t \in [0,\,T^\ast]$. We use a continuation argument to prove that
	\begin{equation*}
		\|\jbl D_x \jbr^\nu e^{\delta^\ast \eta(\jbl D_x \jbr)} \widetilde H^{-1}(t,\,x,\,D_x) U(t,\,\cdot)\|_{L^2}^2 \leq C < \infty,
	\end{equation*}
	for $t \in [0,\,T]$.
	For this purpose, we choose $\delta^\ast < \min\{\delta_0,\,\delta_1,\,\delta_2\}$, and observe that, by the definition of $V(t,\,x)$,
	\begin{equation*}
		\begin{aligned}
			\|\jbl D_x\jbr^\nu e^{\delta^\ast \eta(\jbl D_x \jbr)} \widetilde H^{-1}(t,\,x,\,D_x) U(t,\,\cdot)&\|_{L^2}^2 \\&\leq C \|e^{\delta^\ast \eta(\jbl D_x \jbr)} e^{-\rho(t) \omega(\jbl D_x \jbr)} V(t,\,\cdot)\|_{L^2}^2.
		\end{aligned}
	\end{equation*}
	Next, we use that
	\begin{equation}\label{SH:PROOF:AUXCP:WEAK:WL2}
	\begin{aligned}
	\|e^{\delta^\ast \eta(\jbl D_x \jbr)} e^{-\rho(t) \omega(\jbl D_x \jbr)} V(t,\,\cdot)\|_{L^2}^2
	\leq C \|e^{\delta^\ast \eta(\jbl D_x \jbr)} e^{-\rho(0) \omega(\jbl D_x \jbr)} V(0,\,\cdot)\|_{L^2}^2\\
	+ C \int\limits_0^t \| e^{\delta^\ast \eta(\jbl D_x \jbr)} \jbl D_x \jbr^{\nu} e^{\rho(s) \omega(\jbl D_x\jbr)} \widetilde H^{-1}(s,\,x,\,D_x) F(s,\,\cdot)\|_{L^2}^2 \rmd s,
	\end{aligned}
	\end{equation}
	which can be proved by applying $e^{\delta^\ast \eta(\jbl D_x \jbr)}$ to both sides of \eqref{SH:PROOF:AUXCP:WEAK:CP}. Indeed, we have
	\begin{equation*}
		\begin{aligned}
			\partial_t \Big(e^{\delta^\ast \eta(\jbl D_x \jbr)}V\Big)&=	 e^{\delta^\ast \eta(\jbl D_x \jbr)} \Lambda(t,\,x,\,D_x) e^{-\delta^\ast \eta(\jbl D_x \jbr)}e^{\delta^\ast \eta(\jbl D_x \jbr)}V\\ &+ e^{\delta^\ast \eta(\jbl D_x \jbr)} \I \jbl D_x \jbr^{\nu}e^{\rho(t) \omega(\jbl D_x \jbr)} \widetilde H^{-1}(t,\,x,\,D_x) F(t,\,x),
		\end{aligned}
	\end{equation*}
	where $\Lambda(t,\,x,\,D_x) = \I \widetilde A(t,\,x,\,D_x) - \I \check B(t,\,x,\,D_x) - \kappa \omega(\jbl D_x \jbr)$.
	Again, we use assumptions~\ref{SH:MR:THEOREM:WEAK:ASSUME:Eta} and \ref{SH:MR:THEOREM:WEAK:ASSUME:ConEtaK} to apply Proposition~\ref{APP:PSEUDO:CALC:CON} and Proposition~\ref{APP:PSEUDO:CALC:CON:EstR} and obtain that
	\begin{equation*}
		e^{\delta^\ast \eta(\jbl D_x \jbr)} \Lambda(t,\,x,\,D_x) e^{-\delta^\ast \eta(\jbl D_x \jbr)} \sim \Lambda(t,\,x,\,D_x) - C \delta^\ast \eta(\jbl D_x \jbr).
	\end{equation*}
	We use \eqref{SH:MR:THEOREM:WEAK:MainTheorem:Equ} to satisfy the hypothesis of the sharp form of G{\aa}rding's inequality without any restrictions on $t$, more precisely, we have
	\begin{equation*}
		\begin{aligned}
			\Re[\I \check B(t,\,x,\,\xi) + \kappa \omega(\jxi) &+ C \delta^\ast \eta(\jxi)]\\ &\geq \Re[- C_B(1+\rho(t)) \omega(\jxi) + C \delta^\ast \eta(\jxi)]\\
			&= \Re\bigg[\delta^\ast \eta(\jxi)\Big( C- \frac{C_B(1+\rho(t))}{\delta^\ast} \frac{\omega(\jxi)}{\eta(\jxi)}\Big)\bigg]\geq 0,
		\end{aligned}
	\end{equation*}
	for sufficiently large $|\xi|$.
	Application of Gronwall's Lemma then yields \eqref{SH:PROOF:AUXCP:WEAK:WL2}.
	
	Lastly, we find that
	\begin{equation*}
		\begin{aligned}
			\|e^{\delta^\ast \eta(\jbl D_x \jbr)} e^{-\rho(0) \omega(\jbl D_x \jbr)} V(0,&\,\cdot)\|_{L^2}^2 \\&\leq C\|\jbl D_x\jbr^s e^{\delta^\ast \eta(\jbl D_x \jbr)} \widetilde H^{-1}(0,\,x,\,D_x) U(0,\,\cdot)\|_{L^2}^2,
		\end{aligned}
	\end{equation*}
	which is bounded due to \ref{SH:MR:THEOREM:WEAK:ASSUME:DATA}.
	From these observations, we conclude that $U(T^\ast,\,\cdot) \in H^{\nu}_{\eta,\,\delta^\ast}$.
	
	Since $U(T^\ast,\,\cdot) \in H^{\nu}_{\eta,\,\delta^\ast}$, we are able to prove well-posedness of problem~\eqref{SH:PROOF:AUXCP:WEAK:CP} for $t \in [T^\ast,\,2 T^\ast]$ and may conclude that $U(2 T^\ast,\,\cdot) \in H^{\nu}_{\eta,\,\delta^{\ast\ast}}$, by using the same argument as above, where $\delta^{\ast\ast} < \delta^\ast$. Iteration of this procedure then yields well-posedness for all times $t \in [0,\,T]$ and we may conclude that the solution $U = U(t,\,x)$ to problem~\eqref{SH:PROOF:FAC:System} belongs to $C([0,\,T];\,H^\nu_{\eta,\,\delta})$, where $\delta < \min\{\delta_0,\,\delta_1,\,\delta_2\}$.
	Since $U(t,\,x) = (u_0(t,\,x) ,\,\ldots,\,u_{m-1}(t,\,x))^T$, with
	\begin{equation*}
		u_j(t,\,x) = \jbl D_x\jbr^{m-1-j} (D_t - \lambda_j(t,\,x,\,D_x)) \ldots (D_t - \lambda_1(t,\,x,\,D_x)) u(t,\,x),
	\end{equation*}
	we conclude that the original Cauchy problem \eqref{SH:MR:THEOREM:WEAK:MainTheorem:Equ} has a unique global (in time) solution $u=u(t,\,x)$, with
	\begin{equation*}
		u \in \bigcap\limits_{j = 0}^{m-1} C^{m-1-j}([0,\,T];\, H^{\nu+j}_{\eta,\,\delta}),
	\end{equation*}
	where $\delta < \min\{\delta_0,\,\delta_1,\,\delta_2\}$.
	The proof of Theorem~\ref{SH:MR:THEOREM:WEAK:MainTheorem} is complete.
	

	
\end{document}